\pdfoutput=1
\pdfoutput=1
\pdfoutput=1
\pdfoutput=1
\pdfoutput=1
 \documentclass[a4paper,10pt]{article}
 \usepackage{calc}
 \usepackage[all]{xy}
 \usepackage[centertags]{amsmath}
 \usepackage{latexsym}
 \usepackage{amsfonts}
 \usepackage{graphicx}
 \usepackage{tikz}
 \usepackage{cases}
 \usepackage{amssymb}
 \usepackage{amsthm}
 \usepackage{color}
 \usepackage{fancyhdr}
 \usepackage[dvips]{epsfig}
 \usepackage{newlfont}
 \usepackage[latin1]{inputenc}
 \usepackage[latin1]{inputenc}
 \usepackage[english,french]{babel}
 \usepackage{graphicx}
 \usepackage{t1enc}
 \usepackage[english,french]{babel}
 \usepackage{fancybox}
 \usepackage{graphicx}
 \usepackage{t1enc}
 \usepackage[french2]{minitoc}
 \usepackage{mathrsfs}
 \usepackage{amsfonts}
 \usepackage{wasysym}
 \usepackage{hyperref}
 \allowdisplaybreaks
 \usepackage{float}
\numberwithin{equation}{section} 
\theoremstyle{plain} 
\newtheorem{theo}{Theorem}[]
\newtheorem{theodefi}{Theorem Definition}[section]
\newtheorem{lem}[theodefi]{Lemma}
\newtheorem{rem}[theodefi]{Remark}

\newtheorem{prop}[theodefi]{Proposition}

\newtheorem{fact}[theodefi]{Fact}

\newtheorem{defi}[theodefi]{Definition}

\newtheorem{cor}[theodefi]{Corollary}

\newtheorem{nota}[theodefi]{Notation}
\newtheorem{nota def}[theodefi]{Notations and Definitions}

\newcommand{\po}{{\textbf{Poin}}}

\newcommand{\cro}{{\textbf{Cr}}}
\newcommand{\cri}{{\textbf{CRI}}}
\newcommand{\D}{{ \mathcal D}}
\newcommand{\eg}{{{\ell_1} }}
\newcommand{\ed}{{{\ell_2} }}

\newcommand{\RN}{{\mathcal R}}

\newcommand{\R}{{\mathbb R}}
\newcommand{\Ce}{{\mathcal S}^1}
\newcommand{\N}{{\mathbb N}}
\newcommand{\Z}{{\mathbb Z}}

  \makeatletter
 \newcommand{\thechapterwords}
 { \ifcase \thechapter\or 1\or 2\or 3\or 4\or 5\or
 	6\or 7\or 8\or 9\or 10\or 11\fi}
 \def\thickhrulefill{\leavevmode \leaders \hrule height 2ex \hfill \kern \z@}
 \def\@makechapterhead#1{%
 	\vspace*{15\p@}%
 	{\parindent \z@ \centering \reset@font
 		\thickhrulefill\quad
 		\scshape  {\chapnumfont \@chapapp{}}{\chapnumfont \thechapterwords}
 		\quad \thickhrulefill
 		\par\nobreak
 		\vspace*{15\p@}%
 		\interlinepenalty\@M
 		\hrule
 		\vspace*{15\p@}%
 		\huge {\bfseries  #1}\par\nobreak
 		\par
 		\vspace*{15\p@}%
 		\hrule
 		\vskip 15\p@
 	}}
 	\def\@makeschapterhead#1{%
 		\vspace*{15\p@}%
 		{\parindent \z@ \centering \reset@font
 			\thickhrulefill
 			\par\nobreak
 			\vspace*{15\p@}%
 			\interlinepenalty\@M
 			\hrule
 			\vspace*{15\p@}%
 			\Huge \bfseries #1\par\nobreak
 			\par
 			\vspace*{15\p@}%
 			\hrule
 			\vskip 30\p@
 		}}
 		
 		\DeclareFixedFont{\chapnumfont}{T1}{phv}{b}{n}{20pt}
 		\DeclareFixedFont{\chapchapfont}{T1}{phv}{b}{n}{16pt}
 		\DeclareFixedFont{\chaptitfont}{T1}{phv}{b}{n}{24.88pt}
 		\def\@makechapterhead#1{%
 			\vspace*{15\p@}%
 			{\parindent \z@ \centering \reset@font
 				\thickhrulefill\quad
 				\scshape {\chaptitfont\color[rgb]{0.00,0.50,1.00}\@chapapp{}}
 				{\chapnumfont \thechapterwords}
 				\quad \thickhrulefill
 				\par\nobreak
 				\vspace*{15\p@}%
 				\interlinepenalty\@M
 				\hrule
 				\vspace*{15\p@}%
 				{\Large\bfseries #1}\par\nobreak
 				\par
 				\vspace*{15\p@}%
 				\hrule
 				\vskip 30\p@
 			}}%
 			
\begin{document}
 				\selectlanguage{english}
 	\title{ Rigidity of Fibonacci Circle Maps with a Flat Piece and  Different Critical Exponents}
 		\author{ \large{TANGUE NDAWA Bertuel }
 			\\
	bertuelt@yahoo.fr\\https://orcid.org/0000-0001-8995-9522
	\vspace{0.5cm}
 		\\\today}
				\date{ }
 				\maketitle
		\selectlanguage{english}

 \paragraph{Abstract}
 We consider order preserving $C^3$ circle maps with a flat piece, Fibonacci rotation number, critical exponents $(\ell_1, \ell_2)$ and negative shwarzian derivative.
 			 
 This paper treat the geometry characteristic of the non-wondering  (cantor (fractal)) set from a map of our class. We prove that, for  $(\ell_1, \ell_2)$ in   $(1,2)^2$, the geometry of system is degenerate (double exponentially
 fast). As consequences,  the  renormalization diverges and the geometric (rigidity) class depends  on the three  couples $(c_u(f), c'_u(f) )$, $( c_+(f), c'_+(f))$ and $(c_s(f), c'_s(f) )$.\vspace{0.5cm}
 

  
\textbf{\underline{Key words}}: Circle map, Flat piece, Critical exponent, Geometry, Renormalization and Rigidity.\vspace{0.5cm}

	\section{Introduction}
	We are interested at in certain class of weakly order preserving, non-injective (on an interval exactly; called flat piece) circle maps which appear naturally in the study of Cherry flows on the two dimensions torus (see \cite{MSM},\cite{Livi2013}), non-invertible circle continuous  maps (see \cite{Mi}) and of the dependence of the rotation interval on the parameter value for one-parameter families of circle continuous   maps (see \cite{2}). We write $S^1=\R/\Z$ for the circle. There is the natural projection $\pi: \R\longrightarrow S^1.$ This provides a unique (up to integer translation) lift of a map $f$ of our class to a real continuous map $F$.	An important characteristic of such a map is rotation number defined as follows: 
\begin{equation*}
\rho (f):=\lim_{n\rightarrow\infty }\dfrac{F^n(x)-x}{n}(mod\:1).
\end{equation*}
When the rotation number $\rho (f)$ is irrational (which is more interesting in the study of dynamics, see \cite{de Melo and van Strien} (p.19-36). Also, a map of our class is infinitely renormalizable when its rotation number is irrational, see point 1 of \textbf{Remark~\ref{nth renormalization}}), the denominators of the nearest rational approximants of $\rho (f)$ are defined as follows: 
 $$q_0=1, \;q_1=a_0 \mbox{ and } q_{n+1}=a_nq_{n}+q_{n-1} ,\quad n\geq 2$$ with 
\begin{equation*}\label{i1}
\rho (f)=[a_0a_1\cdots ]:=\dfrac{1}{a_0+\dfrac{1}{a_1+\dfrac{1}{ \ddots}}};\quad 1\leq a_i<\infty.
\end{equation*}
We say that the rotation number $\rho (f)$ is Fibonacci or golden mean if
$a_n=1$ for all $n\in\N$. A circle map with a Fibonacci rotation number will simply be called a Fibonacci circle map. For technical reasons (on the asymptotic behavior of the renormalization operator) the main theorem is proved for Fibonacci circle maps with a flat piece. What prompted the title of this paper.

The purpose of this paper is to describe the rigidity (characteristic geometric) of certain weakly order preserving circle maps with a flat interval and Fibonacci rotation number. In the theory of low-dimensional dynamics, rigidity consists to study (specify) the geometry class under topological class. Being given a such map $f$, its topological class is defined by: $\{h\circ f\circ h^{-1},\; h \mbox{ is homeomorphism}\} $.
For circle maps with a flat interval, geometry depends on the so-called critical exponents $(\ell_1,\ell_2)$  the degrees of the singularities close to the boundary points of the flat interval; the geometry class in the symmetric case $(\ell,\ell)$ with $\ell\in (1,2)$ is treated in \cite{ML}. For a fixed irrational rotation number $\rho$, there is exactly one topological class for circle map with a flat piece. More precisely, two continuous circle maps $f$ and $g$  with a flat piece and same irrational rotation number  are topologically conjugate (that is,  there exists a homeomorphism $h$ such that $h\circ f=g\circ h$), see \cite{MSM}. $h$ is called the topological conjugation or conjugacy between $f$ and $g$.    Let us note that, if the choice of the topological conjugation could be arbitrary inside 
$U_f$ (the flat piece of $f$; that is, $f(U_f)$ is a point), it still uniquely defined on   $K_f:=\Ce\setminus \bigcup _{i=0}^{\infty} f^{-i}(U_f)$, the attractor of $f$. As a consequence, the   class of  regularity of $h$ on $K_f$ is  very interesting. The geometry of $K_f$ and $K_g$ ($f$ and $g$) are  closed respect to this regularity.  In particular,  when  $h$ is $C^{1+\beta},\;\beta>0 $ diffeomorphism, $K_f$ and $K_g$ have the same geometry, the geometry of two systems is rigid, it is not possible to modify it on asymptotical small scales. This result is very surprising according to the class of conjugacy. Indeed, the rigidity has been studied for circle diffeomorphisms \cite{Yo} \cite{He}, critical circle homeomorphisms \cite{FM}, \cite{Ya}, \cite{GMM},  \cite{AV} unimodal maps \cite{Mc}, \cite{MP}, \cite{FMP}, circle maps with breakpoints \cite{KK}, \cite{KT}, for Kleinian groups \cite{MO}. And, in all these cases, it turned out that  the conjugations are  differentiable. 

Before continuing to explain our results, we define our class and fix some notations.

\paragraph{The class of functions:}
We consider $S^1$  as an interval $[a, 1]$ where we identify $a$ with 1.
We consider the class $\mathscr{L}^{X} $ described as follows:

We fix $\ell_1, \ell_2>1$ and denote:\\
-by  $\Sigma^{X} $ the set
\begin{equation*}
\Sigma^{X}=\{(x_1,x_2,x_3,x_4,s)\in \R^5|
x_1<0,\;x_3<x_4<1,\;0<x_2,\;s<1\},
\end{equation*}
- for $r\in\N$, by $Diff^r([0,1])$ the space of $c^r$ orientation
preserving
diffeomorphisms of $[0,1]$. 

The space of $C^3$ circle maps with a flat interval is denoted by:
\begin{equation*}
\mathscr{L}^{X}=\Sigma^{X}\times(Diff^3([0,1]))^3.
\end{equation*}
A point
\begin{equation*}
f:=(x_1,x_2,x_3,x_4,s,\varphi,\varphi^l,\varphi^r)\in
\mathscr{L}^{X}
\end{equation*}
is defined as follows:\\
$f:[x_1,1]\longrightarrow[x_1,1]$ (by identifying $x_1$ with 1)
\begin{equation*}
f(x)=\begin{cases}
\begin{array}{l}f_-(x)= \begin{array}{lcl}f_1(x)=(1-x_2)q_s\circ\varphi\left(\dfrac{x_1-x}{x_1}\right) +x_2 & \mbox{if} & x\in [x_1,0[\end{array}\\
f_+(x)= \begin{cases} \begin{array}{lcl}f_2(x)=
x_1\left(\varphi^{l}\left(\dfrac{x_3-x}{x_3}\right)\right)^{\ell_1}  & \mbox{if} & x\in ]0,x_3]\\f_3(x)=
0 & \mbox{if} & x\in [x_3,x_4]\\f_4(x)=
x_2\left(\varphi^{r}\left(\dfrac{x-x_4}{1-x_4}\right)\right)^{\ell_2 }&  \mbox{if} & x\in [x_4,1]\end{array}
\end{cases}
\end{array}
\end{cases}
\end{equation*}
with 
\begin{equation*}
q_s (x)= \dfrac{[(1-s)x+s]^{\ell_2}-s^{\ell_2} }{1-s^{\ell_2}}.
\end{equation*}
Note that the dynamics of $f_-$ and $f_+$ are respectively defined
by $f^{q_0}$ and $f^{q_1}$.

These kinds of maps can be represented as follows:
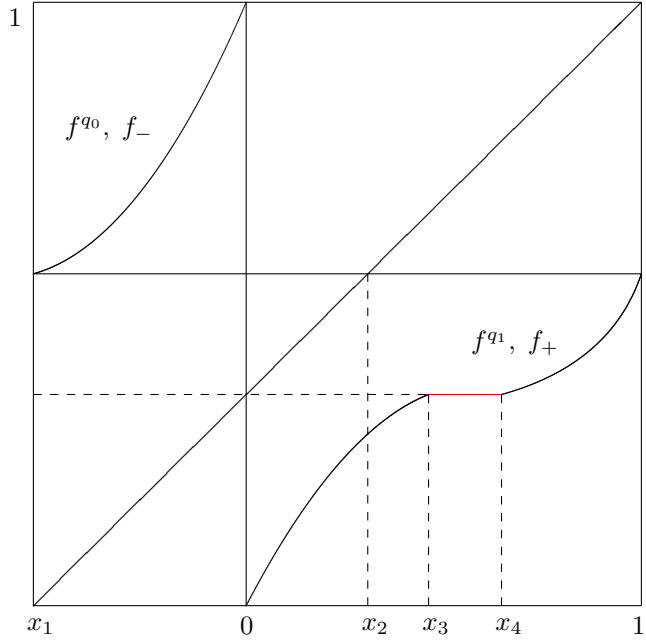
\begin{figure}[H]	
	\centering
	\setlength{\unitlength}{8cm}
	\begin{picture}(1, 1)
	\put(0,0){\line(0,1){1}}
	\put(0,0){\line(1,0){1}}
	\put(1,0){\line(0,1){1}}
	\put(0,1){\line(1,0){1}}
	\put(0,0){\line(1,1){1}}
	\put(0,0.55){\line(1,0){1}} 
	\put(0.35,0){\line(0,1){1}} 
	
	\multiput(0.55,0)(0,0.02825){20}
	{\line(0,1){0.013755}} 
	
	\qbezier(0.35,0)(0.5,0.29)(0.65,0.35)
	\put(0.65,0.35){\color{red}{\line(1,0){0.12}}} 
	\qbezier(0.77,0.35)(0.95,0.4)(1,0.55)
	\put(0.052,0.78){$f^{q_0},\;f_- $}
	\put(0.72,0.425){$f^{q_1},\;f_+ $}
	\multiput(0,0.35)(0.0275,0){24}
	{\line(1,0){0.013755}}
	\multiput(0.65,0)(0,0.02795){13}
	{\line(0,1){0.013755}} 
	\multiput(0.77,0)(0,0.02795){13}
	{\line(0,1){0.013755}} 
	\qbezier(0,0.55)(0.18,0.6)(0.35,1)
	\put(-0.01,-0.04){$x_1$}
	\put(0.34,-0.0475){$0$}
	\put(0.54,-0.04){$x_2$}
	\put(0.64,-0.04){$x_3$}
	\put(0.76,-0.04){$x_4$}
	\put(0.985,-0.0475){$1$}
	\put(-0.041,0.96){$1$}
	
	\end{picture}
	
	\vspace{0.5cm}
	
	\caption{A map $f$ in $\mathscr{L}^{X}$ with $x_2<x_3$} 
	\label{a map 1}    
\end{figure}
or
\begin{figure}[H]
	
	\centering
	\setlength{\unitlength}{8cm}
	\begin{picture}(1, 1)
	\put(0,0){\line(0,1){1}}
	\put(0,0){\line(1,0){1}}
	\put(1,0){\line(0,1){1}}
	\put(0,1){\line(1,0){1}}
	\put(0,0){\line(1,1){1}}
	\put(0.35,0){\line(0,1){1}}
	
	\put(0.052,0.9){$f^{q_0},\;f_- $}
	\put(0.6,0.45){$f^{q_1},\;f_+ $}
	
	\put(0,0.8)	{\line(1,0){1}}  
	
	\qbezier(0,0.8)(0.25,0.85)(0.35,1)
	
	\qbezier(0.67,0.35)(0.90,0.48)(1,0.805)
	\multiput(0.8,0)(0,0.0272){30}
	{\line(0,1){0.013755}} 
	
	\multiput(0,0.35)(0.02775,0){20}
	{\line(1,0){0.013755}}  
	\put(0.55,0.35){\color{red}{\line(1,0){0.12}}}  
	
	\multiput(0.54,0)(0,0.0275){13}
	{\line(0,1){0.013755}} 
	\multiput(0.67,0)(0,0.0275){13}
	{\line(0,1){0.013755}} 
	
	\qbezier(0.35,0)(0.5,0.35)(0.55,0.35)
	
	\put(-0.01,-0.04){$x_1$}
	\put(0.3325,-0.0425){$0$}
	\put(0.7875,-0.04){$x_2$}
	\put(0.5275,-0.04){$x_3$}
	\put(0.6525,-0.04){$x_4$}
	\put(0.985,-0.0425){$1$}
	\put(-0.035,0.975){$1$}
	
	\end{picture}
	\vspace{0.5cm}
	\caption{A map $f$ in $\mathscr{L}^{X}$  with $x_4<x_2$} 
	\label{a map 2}    
\end{figure}
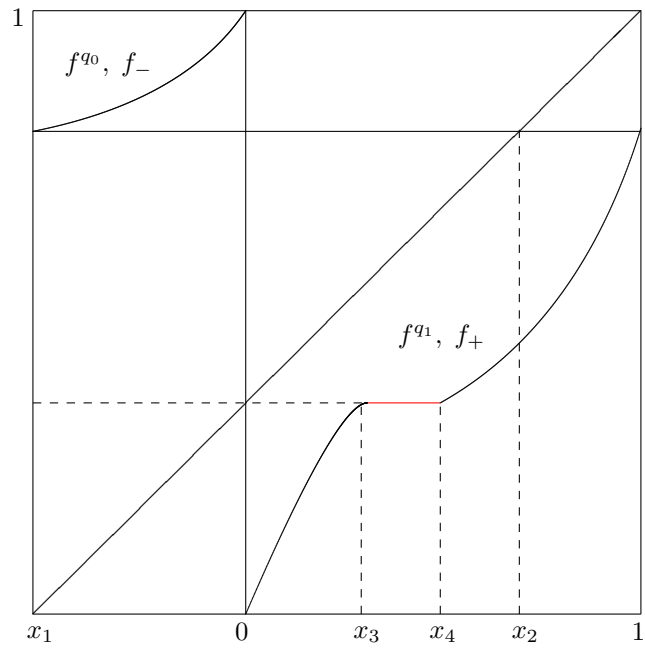

The different branches of a map $f$ in  $\mathscr{L}^{X}$ are described as follows:

Let $[a,b]$ be an interval. 	$A_{[a,b]}: [a,b]\longrightarrow [0,1] $ is the map defined by	$x \longmapsto{(x-a)}/{(b-a)}$ and 	$B_{[a,b]}: [a,b]\longrightarrow [0,1] $ is the map	defined	$x \longmapsto{(x-b)}/{(a-b)}$. 
	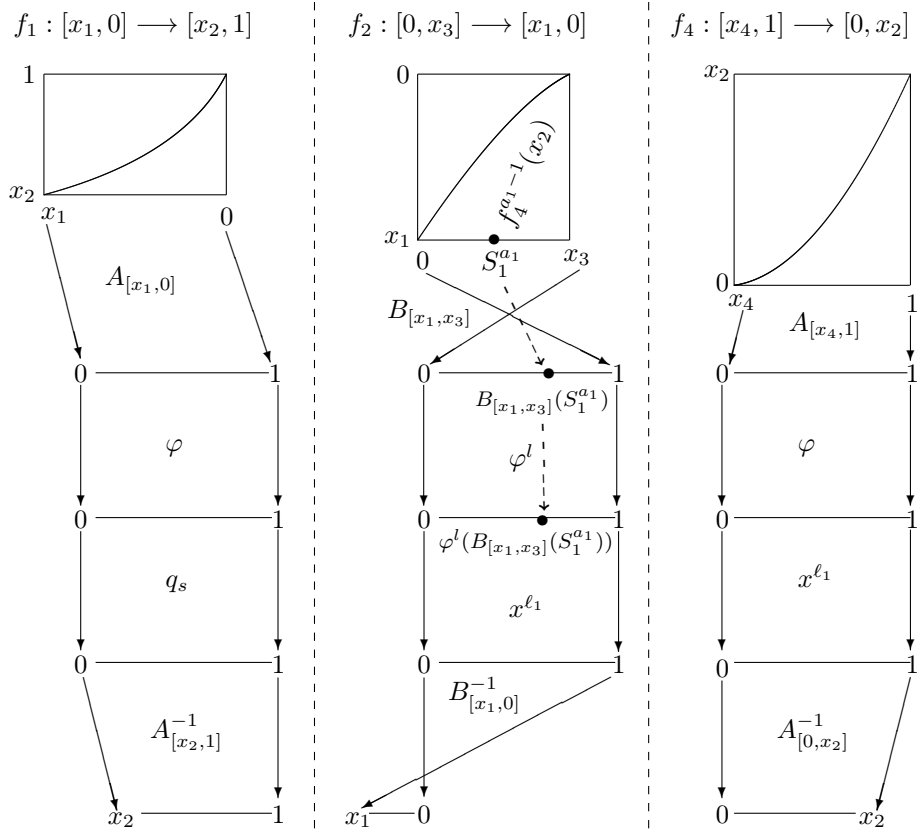
\begin{figure}[H]
	\centering
			\vspace{1cm}
	\setlength{\unitlength}{8cm}

	\hspace{-0.25cm}	
	\begin{picture}(1.25, 1.25)
	\begin{picture}(0.1, 1.25)
	\put(-0.1,1.32){$f_1: [x_1,0]\longrightarrow [x_2,1]$}
	\multiput(0.395,0)(0,0.02825){49}
	{\line(0,1){0.013755}} 
	\put(-0.05,1.25){\line(1,0){0.3}} 
	\put(-0.05,1.05){\line(1,0){0.3}}  
	\put(-0.05,1.05){\line(0,1){0.2}} 
	\put(0.25,1.05){\line(0,1){0.2}} 
	\qbezier(-0.05,1.05)(0.18,1.1125)(0.25,1.25)
	\put(-0.085,1.235){$1$}
	\put(-0.1075,1.045){$x_2$} 
	\put(0.24,0.999995){0}
	\put(-0.055,1.01){$x_1$}
	\put(-0.045,1){\vector(1,-4){0.055}} 
	\put(0.25,0.99){\vector(1,-3){0.0725}} 
	\put(0.0,0.74){0}
	\put(0.32,0.74){1}
	\put(0.035,0.755){\line(1,0){0.285}} 
	\put(0.05,0.9){$A_{[x_1,0]}$}
	\put(0.0105,0.735){\vector(0,-1){0.2}} 
	\put(0.335,0.735){\vector(0,-1){0.2}} 
	\put(0.0,0.4975){0}
	\put(0.325,0.4975){1}
	\put(0.035,0.515){\line(1,0){0.29}} 
	\put(0.15,0.625){$\varphi$}
	\put(0.0105,0.4925){\vector(0,-1){0.1975}} 
	\put(0.335,0.495){\vector(0,-1){0.2}} 
	\put(0.0,0.2575){0}
	\put(0.325,0.2575){1}
	\put(0.035,0.275){\line(1,0){0.29}} 
	\put(0.15,0.4){$q_s$}
	\put(0.015,0.25){\vector(1,-4){0.055}} 
	\put(0.335,0.25){\vector(0,-1){0.2}} 
	\put(0.055,0.008){$x_2$}
	\put(0.11,0.025){\line(1,0){0.215}} 
	\put(0.325,0.009){1}
	\put(0.125,0.15){$A^{-1}_{[x_2,1]}$}
	\end{picture}
	
	
	\hspace{-0.4cm}\begin{picture}(0.1, 1.25)	
	\put(0.385,1.32){$f_2: [0,x_3]\longrightarrow [x_1,0]$}
	\multiput(0.88,0)(0,0.02825){49}
	{\line(0,1){0.013755}} 
	\put(0.5,1.25){\line(1,0){0.25}} 
	\put(0.5,0.975){\line(1,0){0.25}}  
	\put(0.5,0.975){\line(0,1){0.275}} 
	\put(0.75,0.975){\line(0,1){0.275}} 
	\qbezier(0.5,0.975)(0.65,1.2)(0.75,1.25)
	\put(0.445,0.97){$x_1$}
	\put(0.74,0.94){$x_3$}  
	\put(0.465,1.235){$0$} 		
	\put(0.6,0.98){ \rotatebox{70}{{ $f_4^{a_1-1}(x_2)$}}}
	\put(0.6,0.965){ $ \bullet$}
	\put(0.59,0.925){ $ S_1^{a_1}	$}
	\put(0.4975, 0.9275){0}
	\put(0.6275,0.825){\rotatebox{115}{ - - - } }
	\put(0.655,0.77){ \rotatebox{115}{$\dashleftarrow$}}
	\put(0.515,0.92){\rotatebox{7}{\vector(3,-2){0.28}}} 
	\put(0.5925,0.9255){\rotatebox{-39}{\vector(-1,-3){0.092}}} 
	\put(0.5,0.74){0}
	\put(0.82,0.74){1}
	\put(0.535,0.755){\line(1,0){0.285}} 
	\put(0.45,0.845){$B_{[x_1,x_3]}$}
	\put(0.69,0.7425){ $ \bullet$}
	\put(0.58,0.7){ {\footnotesize $ B_{[x_1,x_3]}(S_1^{a_1})	$}}
	\put(0.696,0.585){\rotatebox{92}{ - - - } }
	\put(0.6825,0.525){ \rotatebox{92}{$\dashleftarrow$}}
	
	\put(0.51,0.735){\vector(0,-1){0.2}} 
	\put(0.8275,0.735){\vector(0,-1){0.2}} 
	\put(0.5,0.4975){0}
	\put(0.82,0.4975){1}
	\put(0.68,0.5){ $ \bullet$}
	\put(0.52,0.465){ {\footnotesize $ \varphi^l(B_{[x_1,x_3]}(S_1^{a_1}))	$}}
	
	\put(0.535,0.515){\line(1,0){0.29}} 
	\put(0.65,0.6){$\varphi^l$}
	\put(0.5105,0.4925){\vector(0,-1){0.1975}} 
	\put(0.8305,0.4925){\vector(0,-1){0.2}} 
	\put(0.5,0.2575){0}
	\put(0.82,0.2575){1}
	\put(0.535,0.275){\line(1,0){0.29}} 
	\put(0.65,0.35){$x^{\ell_1}$}
	
	\put(0.5105,0.25){\vector(0,-1){0.1975}} 
	\put(0.8175,0.251){\rotatebox{14}{\vector(-4,-1){0.45}}} 
	\put(0.38,0.008){$x_1$}
	\put(0.42,0.025){\line(1,0){0.075}} 
	\put(0.5,0.009){0}
	\put(0.55,0.2125){$B^{-1}_{[x_1,0]}$}
	\end{picture}
	
	
	\hspace{-1.35cm}	\begin{picture}(0.1, 1.25)	
	\put(0.955,1.32){$f_4: [x_4,1]\longrightarrow [0,x_2]$}
	\put(1.06,1.25){\line(1,0){0.289}} 
	\put(1.06,0.9){\line(1,0){0.289}} 
	\put(1.06,0.9){\line(0,1){0.35}} 
	\put(1.35,0.9){\line(0,1){0.35}} 
	\qbezier(1.06,0.9)(1.2,0.92)(1.35,1.25)
	\put(1.01,1.245){$x_2$}
	\put(1.05,0.865){$x_4$} 
	\put(1.03,0.89){0}
	\put(1.3425,0.855){$1$}
	\put(1.075,0.8575){\vector(-1,-4){0.022}} 
	\put(1.35,0.85){\vector(0,-1){0.075}} 
	\put(1.03,0.74){0}
	\put(1.3425,0.74){1}
	\put(1.06,0.755){\line(1,0){0.285}} 
	\put(1.15,0.825){$A_{[x_4,1]}$}
	\put(1.04,0.735){\vector(0,-1){0.2}} 
	\put(1.35,0.735){\vector(0,-1){0.2}} 
	\put(1.03,0.4975){0}
	\put(1.3425,0.4975){1}
	\put(1.06,0.515){\line(1,0){0.29}} 
	\put(1.165,0.625){$\varphi$}
	\put(1.04,0.4925){\vector(0,-1){0.1975}} 
	\put(1.35,0.495){\vector(0,-1){0.2}} 
	\put(1.03,0.2515){0}
	\put(1.3425,0.2575){1}
	\put(1.06,0.275){\line(1,0){0.29}} 
	\put(1.165,0.4){$x^{\ell_1}$}
	\put(1.35,0.25){\vector(-1,-4){0.055}} 
	\put(1.04,0.245){\vector(0,-1){0.195}} 
	\put(1.265,0.008){$x_2$}
	\put(1.06,0.025){\line(1,0){0.2}} 
	\put(1.03,0.009){0}
	\put(1.13,0.15){$A^{-1}_{[0,x_2]}$}
	\end{picture}


	
\end{picture}

%
%

\caption{Branches of a map $f$ in $\mathscr{L}^{X}$ } 
\label{branches of f}    
\end{figure}

\paragraph{Notations.}
\subparagraph{Common notations.}
	Let $f\in\mathscr{L}^X$ with $U$ as its flat piece.
	\begin{enumerate}
		\item For every $i\in\Z$, instead of $f^i(U)$ we will simply write $\underline{i} $. For example, $\underline{0}=U $. $|\underline{i}| $ stands for the length of the interval $\underline{i} $. So,  $|\underline{i}|  =0$ when $i>0$.  $|(\underline{i},\underline{j} )|$ denotes the distance between the closest endpoints
		of these two intervals while  $|[\underline{i},\underline{j} )|$ stands for $|\underline{i}|+|(\underline{i},\underline{j} )|$. 
		\item Let $I$ be an interval. We write   $ \bar{I} $ to mean the closure of $I$.
		\item The scaling rations $\alpha_n, \,n\in\N$ are defined by 
		\begin{equation*}
		\alpha_n:=\dfrac{|(f^{-q_n}(U), U)|}{|(f^{-q_n}(U), U)|+|f^{-q_n}(U)|}=\dfrac{|(\underline{-q_n}, \underline{0})|}{|[\underline{-q_n}
			, \underline{0})|}, \,n\in\N.
		\end{equation*}
	
		\item 
		For any sequence $\Gamma_n$  and for any  real $d$ we have:
		\begin{equation*}
		\Gamma_n^{d(\ell_1,\ell_2)}:=\begin{cases}\begin{array}{lcl}
		\Gamma_n^{d\ell_1} & \mbox{if}& n\equiv 0[2]\\
		\Gamma_n^{d\ell_2} &\mbox{if} & n\equiv 1[2]
		\end{array}
		\end{cases}  \Gamma_n^{d(\frac{1}{\ell_1},\frac{1}{\ell_2})}:=\begin{cases}\begin{array}{lcl}
		\Gamma_n^{d\frac{1}{\ell_1}} & \mbox{if}& n\equiv 0[2]\\
		\Gamma_n^{d\frac{1}{\ell_2}} &\mbox{if} & n\equiv 1[2].
		\end{array}
		\end{cases}
		\end{equation*}	
		For example, $\alpha_2^{\ell_{1},\ell_{2}}= \alpha_2^{\ell_{1}}$.
		\item Let $x_n$ and $y_n$ be two sequences of positive numbers. We say that $x_n$ is of the order of $y_n$ if there exists a uniform constant $k>0$ such that, for $n$ big enough $x_n<ky_n$. We will use the notation
		\begin{equation*}
		x_n= O(y_n).
		\end{equation*}
	\end{enumerate}	
\subparagraph{Parameters  frequently used.}
	Let $f\in\mathscr{L}^X$  with $\ell_1,\ell_2\in(1,2)$  and Fibonacci rotation number.
\begin{enumerate}
	\item The geometrical (rigidity) characteristics $(c_u(f), c'_u(f) )$, $( c_+(f), c'_+(f))$ and $(c_s(f), c'_s(f) )$ are defined in \textbf{Proposition~\ref{main Proposition}} and \textbf{Lemma~\ref{wn en fct de cu,alpha0}}. Moreover,
	\begin{enumerate}
		\item[-]  $c^{*}_{\iota}(f)=c^{}_{\iota}(f) \mbox{ or }c'_{\iota}(f),\; \iota=s,u,+ $. For more precision, see \textbf{Proposition~\ref{main Proposition}} or \textbf{Notation~\ref{notation cus+}}; 
		\item[-] $\underline{c}_u(f)=\min\{c_u(f),c'_u(f)\}$ and  $\overline{c}_u(f)=\max\{c_u(f),c'_u(f)\}$.		
	\end{enumerate}
	\item For every $n\in\N$, $S_{i,n}$, $y_{i,n}\,i=1,2,3,4,5$ are defined  in the subsection \ref{Sub Asym Renor} respectively in   (\ref{change variable x S}) and  (\ref{change variable S y}).
	\item The vectors $e_i^{\iota}; \,i=2,3,4,5;\,\iota=s,u,+$ are defined in  \textbf{Proposition~\ref{super formular}}. 
	
\end{enumerate}

We consider $f\in\mathscr{L}^X$  with $\ell_1,\ell_2\in(1,2)$  and  Fibonacci rotation number.
The study of the  
renormalization of $f$ generates the geometrical characteristics $(c_u(f), c'_u(f) )$, $( c_+(f), c'_+(f))$ and $(c_s(f), c'_s(f) )$ of $K_f$. This study  depends on the following result:
\begin{lem}\label{main lemma}
	Let  $f $  be a $C^3$ circle  map with a flat piece $U$ and Fibonacci rotation number. The scaling rations $\alpha_{n}$  go to zero double exponentially fast when  $(\ell_1,\ell_2)$ belongs to $(1,2)^2$.
\end{lem}

We will additionally assume in the proof of \textbf{Lemma~\ref{main lemma} }   that,  $f$ has a negative Schwarzian derivative. That means, \begin{equation*}\label{add assumpt 1}
Sf(x):= \dfrac{D^3f(x)}{Df(x)}-\frac{3}{2}\left(\dfrac{D^2f(x)}{Df(x)}\right)^2<0;\quad\forall\:x,\; Df(x)\neq 0.
\end{equation*}

 As a consequence of \textbf{Lemma~\ref{main lemma} }, the renormalization  diverges with three quantitative aspects of the  asymptotic  divergence: an unstable part, a stable part, and a neutral part. The precise formulation is as follows:
\begin{prop}\label{main Proposition}
	Let $(\ell_1,\ell_2)\in (1,2)^2$. There exist $\lambda_u>1$, $\lambda_s\in (0,1)$ such that for $f$ a Fibonacci circle map with a flat piece and critical exponents $(\ell_1,\ell_2)$ the following holds.  There are three geometrical invariant couples  $(c_u(f), c'_u(f) )$, $( c_+(f), c'_+(f))$ and $(c_s(f), c'_s(f) )$ such that for all $n\in \N$
	\begin{equation*}
	\ln\alpha_{n} \asymp c^{*}_u(f) \lambda_u^{p_n}+c^{*}_s(f) \lambda_s^{p_n}+c^{*}_+(f)
	\end{equation*}
	where $p_n$ is the integer part of $n$, $c^{*}_{\iota}(f)=c^{}_{\iota}(f),\; \iota=s,u,+ $ if $n$ is even and $c^{*}_{\iota}(f)=c'_{\iota}(f),\; \iota=s,u,+$ if $n$ is odd. That is, 
	\begin{equation*}
	\ln\alpha_{2p_n} \asymp c_u(f) \lambda_u^{p_n}+c_s(f) \lambda_s^{p_n}+c_+(f)
	\end{equation*}
	and
	\begin{equation*}
	\ln\alpha_{2p_n+1}\asymp c'_u(f) \lambda_u^{p_n}+c'_s(f) \lambda_s^{p_n}+c'_+(f).
	\end{equation*}
\end{prop}


Now, we can formulate the result on rigidity as follows:
\paragraph{Main Theorem:}
Let $(\ell_{1},\ell_{2})\in(1,2)^2$. There exists $\beta=\beta(\ell_{1},\ell_{2})\in(0,1)$ such that 
if   $f$ and $g$ are two circle maps with  Fibonacci rotation number,  critical exponents $(\ell_{1},\ell_{2})$ and if  $h$ is the topological conjugation between $f$ and $g$ then
\begin{align*}
h \mbox{ is Holder homeo};\qquad &  \\
h \mbox{ is a bi-lipschitz homeo}\Longleftrightarrow& c^*_u(f)=c^*_u(g),  c_+(f)+c'_+(f)=c_+(g)+c'_+(g);\\
h   \mbox{ is a } C^{1+\beta} \mbox{  diffeo} \Longleftrightarrow&  c^*_u(f)=c^*_u(g), c_+(f)+c'_+(f)=c_+(g)+c'_+(g),\\  &  c^*_s(f)=c^*_s(g).
\end{align*}

The rigidity  is now completely described.
This paper is organized as follows: in \S \ref{Section 2} and \ref{Section 3}  
we introduce the basic concepts and well-known results used in this paper. \S \ref{Section 4} is devoted to the renormalization under \textbf{Lemma~\ref{main lemma}}: the scaling ratios $\alpha_{n} $ go to zero double exponentially fast when $(\ell_1,\ell_2)$ belongs to $(1,2)^2$. The asymptotic study of $\alpha_{n} $ does not use the main ideas of this paper thus will be postponed until the last section, \S \ref{Section 6}. The main theorem is proved in \S \ref{Section 5}. 


\section{Technical tools}\label{Section 2}

 	
 \begin{rem}~\label{critical exponent and  renormalization}
 	\begin{enumerate}
 		\item The sequences used in this paper are  $((\ell_1, \ell_2),(\ell_1, \ell_2))$ equivalent on the parity in the following sense: Let $\eta_{n}$ be a sequence.  $\eta_{2n}$ dependents on $(\ell_1, \ell_2)$ by a function 
 	$\Psi_{2n}$ if only if $\eta_{2n+1}$ dependents on $(\ell_2, \ell_1)$ by the same function. That is,
 	\begin{equation*}
 	\eta_{2n}=\Psi_{2n}(\ell_1, \ell_2)\Leftrightarrow \eta_{2n+1}=\Psi_{2n+1}(\ell_2, \ell_1).
 	\end{equation*}
 As a consequence, a statement or  proof presented for $n$  even  deduces by himself the case $n$ odd  and vice-versa. 
\item The renormalization operator $\RN$ maps $(\ell_1, \ell_2)$ to $(\ell_2, \ell_1)$. Namely, for a fixed $f\in\mathscr{L}^{X} $ with critical exponents $(\ell_1, \ell_2)$, if $\RN f$ exists, then it belongs to $\mathscr{L}^{X} $ with   critical exponents $(\ell_2, \ell_1)$. Thus, because the asymptotic study of the operator $\RN$ on $f$ is equivalent to that   of $\RN^2$ on $(f,\RN(f) )$, then we are interested at the operator $\RN^2$ which preserves the critical exponents.
\end{enumerate}
 \end{rem}
 
\begin{fact}\label{fact} Let $f\in\mathscr{L}^X$ with $U$ as its flat piece.
	Let $l(U)$ and $r(U)$ be the left and right endpoints of $U$ respectively. There are a left-sided neighborhood $I^{l_1}$ of $l_1(U)$, a	right-sided neighborhood $I^{l_2}$ of $l_2(U)$ and  three positive constants $K_1, K_2, K_3$ such that the following holds: 
	
	\begin{enumerate}
		\item Let  $y\in I^{l_i} $; $i=1,2$. Then
		\begin{equation*}
		\begin{array}{c}	
		K_{1}|l_i(U)-y|^{l_i}  \leq  |f(l_i(U))-f(y) | 	\leq  K_{2}|l_i(U)-y|^{l_i},\\
		K_{1}|l_i(U)-y|^{l_i-1}  \leq \dfrac{df}{dx}(y)  	\leq  K_{2}|l_i(U)-y|^{l_i}.
		\end{array}
		\end{equation*}	
		\item \label{rl4}	Let three points with $y$ between $x$ and $z$ be arranged so that, of the three, the point $z$ is the closest to the flat piece. If the interval $(x,z)$ does intersect the flat piece $U$ (that is, $f$ is a diffeomorphism on $(x,z)$), then
		\begin{equation*}
		\dfrac{|f(x)-f(y)|}{|f(x)-f(z)|}\leq  K_3 \dfrac{|x-y|}{|x-z|}. 
		\end{equation*}.
	\end{enumerate}
\end{fact} 
\begin{rem}~\label{add assump 2}
		
	\begin{enumerate}
\item	The first part of  \textbf{Fact~\ref{fact}} implies that, for $i=1,2$;  $f_{|I^{l_i}}(x) \approx k_i x^{\ell_i} $.
\item Let us note that, \textbf{Fact~\ref{fact}} has also been applied in  control theory, see \cite{YY}. 
\end{enumerate}
\end{rem}

\section{Basic results}\label{Section 3}
 Let $f\in\mathscr{L}^X$ with $U$ as its flat piece.
	\begin{prop}\label{partition of f}
	Let $ n\geq1$.
	\begin{enumerate}
		\item[$ \bullet $] The set of ``long'' intervals consists of the intervals
		\begin{equation*}		
		\mathcal{A}_n:=\{(\underline{q_{n-1}+i+1},\underline{i+1});\;0\leq i <q_{n} \}.
		\end{equation*} 
		\item[$ \bullet $]  The set of ``short'' intervals consists of the intervals
		\begin{equation*}		
		\mathcal{B}_n:=\{(\underline{i+1},i+1+\underline{q_{n}});\;0\leq i <q_{n-1} \}.
		\end{equation*}
	\end{enumerate}
	The set $\mathcal{P}_n:=\mathcal{A}_n\cup \mathcal{B}_n$ covers the circle modulo the end points and is called $n^{th}$ dynamical partition.	The dynamic partition produced by the first $\underline{q_{n+1}+q_n} $ pre-images of $U$ is denoted  $\mathcal{P}^{n} $. It consists of 
	\begin{equation*}		
	\wp_n:=\{\underline{-i};\;0\leq i \leq q_{n-1}+q_{n}-1 \}
	\end{equation*}
	together with the gaps between these sets. As in the case of $\mathcal{P}_n $ there are two kinds of gaps, ``long'' and ``short'':
	\begin{enumerate}
		\item[$ \bullet $] The set of ``long'' intervals consists of the intervals
		\begin{equation*}		
		\mathcal{A}^{n}:=\{ I_i^{n-1}:=(\underline{-i},\underline{-q_{n-1}-i});\;0\leq i < q_{n-1} \}.
		\end{equation*} 
		\item[$ \bullet $]  The set of ``short'' intervals consists of the intervals
		\begin{equation*}		
		\mathcal{B}^{n}:=\{I_i^{n}:=(\underline{-q_{n}-i},\underline{-i});\;0\leq i< q_{n-1} \}.
		\end{equation*}
	\end{enumerate}
\end{prop}
\begin{proof}
This comes from \S 1.4 in \cite{GJSTV}.
\end{proof}
\begin{prop}\label{qn go to zero} 
	The sequence  $|( \underline{0},\underline{q_n})| $
	tends to zero uniformly and at least exponentially fast.
\end{prop}
\begin{prop}\label{preimage and gaps adjacent} 
	If $A$ is a pre-image of $U$ belonging to $\mathcal{P}^{n} $ and if $B$ is one of the gaps adjacent to $A$, then $|A|/|B|$ is bounded away from zero by a constant that does not depend on $n$, $A$ or $B$.
\end{prop}

The proofs of all these results can be found in \cite{GJSTV} (proof of \textbf{Proposition 1} and \textbf{Proposition 2} respectively).

\paragraph{Cross-Ratio Inequality (\cri) }
The cross ratio inequality was introduced and proved by several authors, see \cite{2, de Melo and Van Strien 89, Yo1}. We shall adapt the one introduced in \cite{2} and redefined in \cite{Gra} to our needs. 

Let  $f\in \mathscr{L}^X $ with  $U$ as its flat piece and negative Schwarzian derivative. 	Let $I$ and $J$ be two intervals of finite and nonzero lengths such that $\bar{I}\cap \bar{J}=\emptyset=(\bar{J}\cup\bar{I})\cap  \overline{U}$  and  $J$ is on the right of $I$. We define their cross-ratio as	
	\begin{equation*}
	\cro\:(I,J):=\dfrac{|I||J|}{|[I,J)||(I,J]|}.
	\end{equation*}
	The distortion $\D\cro$\, of the cross-ratio by $f$ is given by
	\begin{equation*}
	\D\cro\:(I,J; f):=\dfrac{\cro\:(f(I),f(J))}{\cro\:(I,J)}.
	\end{equation*} 
	When there is no ambiguity we write $\D\cro\:(I,J)$ instead of $\D\cro\:(I,J; f)$.
	
	Let $n\in \N$ such that	
	\begin{enumerate}
			\item $f^n$ is a diffeomorphism on $(I,J)$;\label{hypothesis 2 of cri}
		\item Each point of the circle belongs to at most $k$ intervals $f^i([I,J])$.
	\end{enumerate}
	Then,
	\begin{equation*}
	\D\cro\:(I,J; f^n)=\prod_{i=0}^{n-1}\D\cro\:(f^i(I),f^i(J); f)<1.\end{equation*}

\begin{rem}\label{diffeo on an interval}
	Let $f\in\mathscr{L}^X$ with $U$ as its flat piece. Given  $n\geq1$,  and  $T$  be an interval. We have
	$f^{n}:T\longrightarrow f^{n}(T)$ is diffeomorphism if only if, for all $0\leq i\leq n-1$, $f^{i}(T)\cap \overline{U}=\emptyset$.
\end{rem}

 	\begin{lem}\label{cn and the other side} 
 	For all $i=1,..., q_{n-1}-1 $, the parameter sequence
 	\begin{equation*}
 	\varrho_n(i):=\dfrac{	|\underline{-q_{n}+i}|}{|(\underline{q_{n-1}+i}, \underline{-q_{n}+i}]|}
 	\end{equation*}
 	is bounded away from zero.
 \end{lem}
 \begin{proof}
 	Observe that	if $0\leq i<q_{n-2}$, then  $\varrho_n(i)$ is  larger than
 	\begin{equation*}
 	\cro([\underline{-q_{n-2}+i}, \underline{q_{n-1}+i}),\underline{-q_{n}+i} )
 	\end{equation*}	
 	which by \cri\, with $f^{q_{n-2}-i} $ is greater than
 	\begin{equation*}
 	\dfrac{|\underline{-q_{n-1}}|}{|(\underline{q_{n}},\underline{-q_{n-1}}]|}
 	\end{equation*}	
 	times a uniform constant. Moreover, the above ratio is bigger than
 	\begin{equation*}
 	\dfrac{|\underline{-q_{n-1}}|}{|(\underline{-q_{n+1}},\underline{-q_{n-1}}]|}
 	\end{equation*}	
 	which  by \textbf{Proposition~\ref{preimage and gaps adjacent}} go away from zero.
 	
 	Now, suppose that $q_{n-2}\leq i:=q_{n-2}+j\leq q_{n-1}-1$. Then,
 	\begin{equation*}
 	\varrho_n(i)\geq\dfrac{	|\underline{-q_{n-1}+j}|}{|(\underline{-q_{n+1}+j}, \underline{-q_{n-1}+j}]|}
 	\end{equation*} 
 	which  by \textbf{Proposition~\ref{preimage and gaps adjacent}} go away from zero.
 	
 This completes the proof.
 \end{proof}
A proof of the following Proposition can be found in \cite{de Melo and van Strien} \textbf{theorem:3.1 p.285}.
\begin{prop}[Koebe principle]\label{koebe principle}
	Let	$ f\in \mathscr{L}^X $. For
	every $\varsigma,\,\alpha>0$ there exists a constant $\zeta(\varsigma,\alpha)>0$ such that  the following holds. Let $T$ and $M\subset T$ be
	two intervals and let $S,\;D$ be the left and the right component
	of $T\setminus M$ and $n\in\N$. Suppose that:
	\begin{enumerate}
		\item $\sum_{i=0}^{n-1}f^i(T)<\varsigma $,
		\item $f^{n}: T\longrightarrow f^n(T)$ is a diffeomorphism,
		\item $\dfrac{|f^n(M)|}{|f^n(S)|}, \;\dfrac{|f^n(M)|}{|f^n(D)|}<
		\alpha.$
	\end{enumerate}
	
	Then,
	\begin{equation*}
	\dfrac{1}{\zeta(\varsigma,\alpha)}\leq\dfrac{	Df^n(x)}{	Df^n(y)}\leq\zeta(\varsigma,\alpha), \quad \forall x,y\in M.
	\end{equation*}
	That is,
	\begin{equation*}
	\dfrac{1}{\zeta(\varsigma,\alpha)}\cdot\dfrac{|A|}{|B|}\leq\dfrac{	f^n(A)}{	f^n(B)}\leq\zeta(\varsigma,\alpha)\cdot\dfrac{|A|}{|B|}, \quad \forall A, B\,\mbox{(intervals)}\subseteq M
	\end{equation*}
	where
	\begin{equation*}
	\zeta(\varsigma,\alpha)=\dfrac{1+ \alpha}{\alpha}e^{C\varsigma}
	\end{equation*}
	and $C\geq 0$ only depends on $f$.
\end{prop}

\section{Renormalization}\label{Section 4}
The key of this part is to set up for $n$ large enough a right algebraic relation between the $(n+1)$th renormalization and  $n$th renormalization of the form: $R^{n+1} f =L R^{n}f +C_0(f)$ with  $f$ a map our class with Fibonacci rotation number. The technique used requires a double change  of variables on $\Sigma^{X} $. Let us mention once again that \textbf{Lemma~\ref{main lemma}} (the scaling ratios $\alpha_{n} $ go to zero double exponentially fast when $(\ell_1,\ell_2)$ belongs to $(1,2)^2$) will be used frequently throughout this part.

\subsection{Renormalization operator}
Let $f\in \mathscr{L}^X$ with $[a_0,a_1,\dots]$ as its rotation number. If $1\leq a_0,a_1 < \infty$, then $f$ can be renormalized in the following sense:
 let  $x_{1,1}=f^{a_1-1}(x_2)/x_1$, let $h: [x_1,f_4^{a_1-1}(x_2))\longrightarrow [x_{1,1},1)$ be the map $x\longmapsto {x}/{x_1} $ and let $\RN f :  [x_{1,1},1)\longrightarrow [x_{1,1},1)$
  be the map defined by 
\begin{equation}\label{renormalization definition}
\RN f (x)=\begin{cases}
\begin{array}{lcl}
\RN f_-(x)=h\circ f_+\circ h^{-1}(x) &\mbox{if} & x \in [x_{1,1},0)\\
\RN f_+(x)=h\circ f_+^{a_1}\circ f_-\circ h^{-1}(x) &\mbox{if} & x \in (0,1)
\end{array}
\end{cases} 
\end{equation} 
More precisely,
\begin{equation*}\label{renormalization definition explicit}
\RN f (x)=\begin{cases}
\begin{array}{lcl}
\RN f_-(x)=h\circ f_2\circ h^{-1}(x) &\mbox{if} & x \in [x_{1,1},0)\\
\RN f_+(x)=h\circ f_2\circ f_3^{a_1-1}\circ f_1\circ h^{-1}(x) &\mbox{if} & x \in (0,1]
\end{array}
\end{cases} 
\end{equation*} 

Observe that $$\RN f=h\circ p\RN f\circ h^{-1}$$ where
\begin{equation*}\label{prerenormalization definition}
p\RN f (x)=\begin{cases}
\begin{array}{lcl}
\RN f_-(x)=f_+(x) &\mbox{if} & x \in (x_1,0)\\
\RN f_+(x)=f_+^{a_1}\circ f_-&\mbox{if} & x \in (0, f_4^{a_1-1}(x_2)]
\end{array}
\end{cases} 
\end{equation*} 
Observe that  $\RN f$ is nothing else than $p\RN$ rescaled and flipped.
The function $h$ is  a normalization operator, $p\RN$ is the first return (Poincaré) map of $f$ to the interval $(x_1, f_4^{a_1-1}(x_2)]$  and $\RN f$ is called the first renormalization of $f$.

  By $\mathscr{L}_0^X $ we denote the subset of $\mathscr{L}^X$ consisting of maps renormalizable.  

 Before explaining more precisely $\RN f$, we introduce the zoom map  on an interval $[a,b]$.
 $Z_{[a,b]}: Diff([0,1]) \longrightarrow  Diff([0,1])$ is the map $\psi \longmapsto A_{[\psi(a),\psi(b)]}\circ\psi\circ
A^{-1}_{[a,b]} $ where		$A_{[a,b]}: [a,b]\longrightarrow [0,1] $ is the map		$x \longmapsto{(x-a)}/{(b-a)}$. 

Now, we are ready to describe more precisely the function $\RN f$ for a fixed $f\in\mathscr{L}_0^X $.
\begin{prop}\label{first renormalization on x}
Let $f\in\mathscr{L}_0^X$ with $[a_0,a_1,\dots]$ as its rotation number.
If $1\leq a_0,a_1\leq\infty$, then $\RN f$ is a map of $\mathscr{L}$. More precisely, 
\begin{enumerate}
	\item  $\rho(\RN f)=[a_1,a_2,\dots]$,
	\item $\RN f:=(x_{1,1} ,x_{2,1} ,x_{3,1} ,x_{4,1} ,s_1,\varphi_1,\varphi^l_1,\varphi^r_1)$	
	with
\begin{align*}
		x_{1,1}&=\dfrac{f_4^{a_1-1}(x_2)}{x_1}, \nonumber\\\vspace{0.1cm}
		x_{2,1}  &=\left[\varphi^l\left(\dfrac{x_3-f_4^{a_1-1}(x_2)}{x_3}\right)\right]^{\ell_1} , \nonumber\\\vspace{0.1cm}
		x_{3,1}&=1-\varphi^{-1}\circ q^{-1}_s\left(\dfrac{f_4^{-a_1+1}(x_4)-x_2 }{1-x_2} \right), \nonumber\\\vspace{0.1cm}
		x_{4,1}  &=1- \varphi^{-1}\circ q^{-1}_s\left(\dfrac{f_4^{-a_1+1}(x_3)-x_2 }{1-x_2}\right), \\\vspace{0.1cm}
		s_1&=\varphi^l\left(\dfrac{x_3-f_4^{a_1-1}(x_2)}{x_3}\right), \nonumber\\\vspace{0.1cm}
		\varphi_1&=Z_{\left[1-f_4^{a_1-1}(x_2)/x_3,1\right]}(\varphi^l), \nonumber \\\vspace{0.1cm}
		\varphi^l_1&=Z_{\left[{(x_2-x_4)}/{(1-x_4)},1\right]}\left(\varphi^r\right)^{a_1-1}\circ Z_{\left[1-x_{3,1},1\right]}(q_s\circ\varphi), \nonumber \\\vspace{0.1cm}
		\varphi^r_1&=Z_{\left[0,\dfrac{x_3-f_4^{a_1-1}(x_2)}{x_3}\right]}(\varphi^l)\circ Z_{\left[\dfrac{x_2-x_4}{1-x_4},1\right]}\left(\varphi^r\right)^{a_1-1}\circ
		Z_{\left[0,1-x_{4,1}\right]}(q_s\circ\varphi).\nonumber
	\end{align*}	
 
\end{enumerate} 

\end{prop}
\begin{proof}~
	\begin{enumerate}
		\item The point 1 is a general result on the renormalization of circle maps.
		\item The  graph of $p\RN$ is described in the quadrant in blue dashes of the following figure.
	
		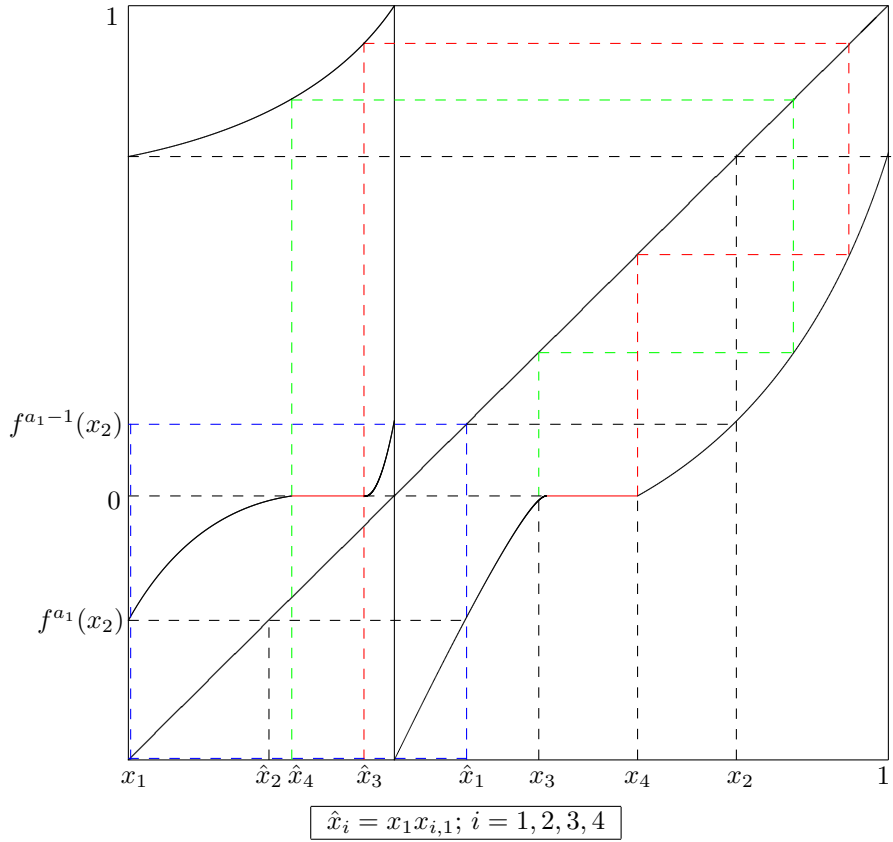
\begin{figure}[H]
			
			\centering
			\setlength{\unitlength}{10cm}	
			\hspace{1cm}
			\begin{picture}(1, 1)
			\put(0,0){\line(0,1){1}}
			\put(0,0){\line(1,0){1}}
			\put(1,0){\line(0,1){1}}
			\put(0,1){\line(1,0){1}}
			\put(0,0){\line(1,1){1}}
			\put(0.35,0){\line(0,1){1}}

			\qbezier(0, 0.185)(0.08,0.33)(0.215,0.35)
			
			\qbezier(0.31,0.35)(0.33,0.345)(0.35,0.45)
			\multiput(0,0.185)(0.0265,0){17}
			{\line(1,0){0.013755}} 
			\multiput(0.185,0)(0,0.027){7}
			{\line(0,1){0.013755}} 
			\put(-0.12,0.175){$f^{a_1}(x_2)$}
			
			
			\put(0.215,0.35){\color{red}{\line(1,0){0.095}}}  
			
			\multiput(0.215,0)(0,0.027){33}
			{\color{green}{\line(0,1){0.013755}}} 
			
			\multiput(0.31,0)(0,0.0268){36}
			{\color{red}{\line(0,1){0.01377}}} 

			\multiput(0.215,0.875)(0.027,0){25}
			{\color{green}{\line(1,0){0.013755}}} 
			\multiput(0.31,0.95)(0.027,0){24}
			{\color{red}{\line(1,0){0.013755}}} 
			
			\multiput(0.875,0.54)(0,0.027){13}
			{\color{green}{\line(0,1){0.013755}}} 
			\multiput(0.948,0.67)(0,0.0267){11}
			{\color{red}{\line(0,1){0.013755}}} 
			
			\multiput(0.54,0.54)(0.029,0){12}
			{\color{green}{\line(1,0){0.013755}}} 
			\multiput(0.67,0.67)(0.029,0){10}
			{\color{red}{\line(1,0){0.014}}} 
			
			\multiput(0,0.8)(0.0275,0){37}
			{\line(1,0){0.013755}}  

			\multiput(0.445,0)(0,0.027){17}
			{\color{blue}{\line(0,1){0.015}}} 
			\multiput(0,0.445)(0.027,0){17}
			{\color{blue}{\line(1,0){0.013755}}} 
			
			\put(-0.1575,0.435){$f^{a_1-1}(x_2)$}

			\qbezier(0,0.8)(0.25,0.85)(0.35,1)
			
			\multiput(0.003,0)(0,0.027){17}
			{\color{blue}{\line(0,1){0.013755}}} 
			\multiput(0,0.002)(0.027,0){17}
			{\color{blue}{\line(1,0){0.013755}}} 

			\multiput(0.445,0.445)(0.0275,0){13}
			{\line(1,0){0.013755}}  

			\multiput(0.54,0.35)(0,0.0255){8}
			{\color{green}{\line(0,1){0.012}}} 
			\multiput(0.67,0.35)(0,0.0255){13}
			{\color{red}{\line(0,1){0.013755}}} 
			\multiput(0.80,0)(0,0.0275){30}
			
			\qbezier(0.67,0.35)(0.90,0.48)(1,0.805)
			\multiput(0.8,0)(0,0.0272){30}
			{\line(0,1){0.013755}} 
			
			\put(0.55,0.35){\color{red}{\line(1,0){0.12}}}  
			
			\multiput(0.54,0)(0,0.0275){13}
			{\line(0,1){0.013755}} 
			\multiput(0.67,0)(0,0.0275){13}
			{\line(0,1){0.013755}} 
			
			\qbezier(0.35,0)(0.52,0.35)(0.55,0.350)
			
			\multiput(0,0.35)(0.0275,0){8}
			{\line(1,0){0.013755}}  
			\multiput(0.31,0.35)(0.0275,0){9}
			{\line(1,0){0.013755}}  

			
			\put(0.16725,-0.03){$\hat{x}_2$}
			\put(0.435,-0.03){$\hat{x}_1$}
			\put(0.21,-0.03){$\hat{x}_4$ }
			\put(0.3,-0.03){$\hat{x}_3$}
			
			
			\put(-0.01,-0.03){$x_1$}
			\put(0.7875,-0.03){$x_2$}
			\put(0.5275,-0.03){$x_3$}
			\put(0.6525,-0.03){$x_4$}

			\put(-0.0275, 0.3325){$0$}
			\put(0.985,-0.0325){$1$}
			\put(-0.03,0.975){$1$}
			
			\end{picture}
			\vspace{0.5cm}\\
			\begin{tabular}{|c|}
				\hline
				$\hat{x}_i=x_1x_{i,1}$; $i=1,2,3,4$\\\hline
			\end{tabular}
			\caption{First return map of a map $f$ in $\mathscr{L}^{X}$} 
			\label{first return map}    
		\end{figure}
		
		By combining \textbf{(\ref{renormalization definition})}, the descriptions of branches of $f$ (see, Figure~\ref{branches of f}) and the graph of $p\RN$ (see, Figure~\ref{first return map}) the point 2 follows.
	\end{enumerate}

\end{proof}

\begin{defi}\label{infintely renormalizable map of our class}
Let $f\in \mathscr{L}^X$. $f$ is infinitely renormalizable if for every $n\in\N$, $R^nf\in \mathscr{L}_0^X$. The class of   infinitely renormalizable functions will be denoted by $\mathcal{W}^X $. 
\end{defi}

The following observation explain more the $n$th renormalization, for every $n$ in $\N$. 
\begin{rem}\label{nth renormalization} Let $f$ be a map in $\mathcal{W}^X $ with rotation number $\rho(f)=[a_0,a_1,\dots]$ and  critical exponents $(\ell_1,\ell_2)$. For every $n\in\N$, we have:
\begin{enumerate}
\item  $\RN^nf$ belongs to $\mathcal{W}^X $ with $\rho(\RN^nf)=[a_n,a_{n+1},\dots]$. As a consequence, $\mathcal{W}^X $ corresponds to the subset of $\mathscr{L}^X$ with irrational rotation number.	
\item   The  critical exponents of  $\RN^nf$ is $(\ell_1,\ell_2)$ if $n$ is even and $(\ell_2,\ell_1)$ if $n$ is odd. For example, if $n$ is even, 
$$	\RN^n f:=(x_{1,n} ,x_{2,n} ,x_{3,n} ,x_{4,n} ,s_n,\varphi_n,\varphi^l_n,\varphi^r_1)\in
\mathcal{W}^{X}$$
and 
\begin{equation*}
\RN^nf(x)=\begin{cases}
\begin{array}{lcl}\RN^nf_-(x) &\mbox{ if } & x\in [x_{1,n},0)\\
\RN^nf_+(x) &\mbox{ if } & x\in (0,1]
\end{array}
\end{cases}
\end{equation*}	
with  
\begin{align}
		\RN^nf_-(x)&=f_{1,n}(x)\nonumber\\\label{explicite renormazation function 0}
		&=(1-x_{2,n})q_{s_n}\circ\varphi_n\left(\dfrac{x_{1,n}-x}{x_{1,n}}\right) +x_{2,n} \mbox{ if }  x\in [x_{1,n},0)
		\end{align}
		and 
		$$\RN^nf_+(x)= \begin{cases} \begin{array}{lcl}f_{2,n}=
		x_{1,n}\left(\varphi_n^{l}\left(\dfrac{x_{3,n}-x}{x_{3,n}}\right)\right)^{\ell_1}  &\mbox{ if }& x\in ]0,x_{3,n}]\\f_{3,n}=
		0&\mbox{ if } & x\in [x_{3,n},x_{4,n}]\\f_{4,n}=
		x_{2,n}\left(\varphi_n^{r}\left(\dfrac{x-x_{4,n}}{1-x_{4,n}}\right)\right)^{\ell_2}&  \mbox{ if }& x\in [x_{4,n},1]\end{array}
		\end{cases}	$$
		
		\item The dynamic of $\RN^nf$ is controlled by the couple $(q_{n-1},q_n)$ in the following sens:	$\RN^nf_{|[x_{1,n}, 0)}$ and $\RN^nf_{|(0,1)}$ are respectively a rescaled version of $f^{q_{n-1}} $ and  $f^{q_{n}} $. Namely, the quadruplet $(x_{1,n}, x_{2,n}, x_{3,n}, x_{4,n} )$ corresponds to 
		$( \underline{q_{n}+1},\underline{q_{n+1}+1}, r(\underline{-q_{n}+1}), l(\underline{-q_{n}+1}) )$ if $n$ is even and corresponds to 
		$( \underline{q_{n}+1},\underline{q_{n+1}+1}, l(\underline{-q_{n}+1}), r(\underline{-q_{n}+1}) )$ if n is odd.
	\end{enumerate}
\end{rem}	
\subsection{Asymptotic Distortions} 
This part is devoted to prove the following result. 
\begin{prop}\label{asymptotic distortion}
Let $ f\in\mathcal{W}^X $.  For every $n$,
\begin{equation*}
dist(\varphi_{n}^l)=O\left(\alpha_{n-1}^{\frac{1}{\ell_{1}}, \frac{1}{\ell_{2}}}\right),\;	dist(\varphi_{n})=O\left(\alpha_{n-2}^{\frac{1}{\ell_{1}}, \frac{1}{\ell_{2}}}\right),\;
dist(\varphi_{n}^r)=O\left(\alpha_{n}^{\frac{1}{\ell_{1}}, \frac{1}{\ell_{2}}}\right)
\end{equation*}
\end{prop}
\begin{proof}
	It is based on the Koebe principle (\textbf{Proposition~\ref{koebe principle}}). 
	\paragraph{Distortion of $\varphi_n^l$: }\label{distortion phi_nl}
	\begin{itemize}
		\item[-] $T=[ \underline{-q_{n+1}+1+(a_n-1)q_n},\underline{-q_n+1}]$,  
		\item[-] $M=( \underline{1},\underline{-q_n+1} ), $ 
		\item[-] $S=[ \underline{-q_{n+1}+1+(a_n-1)q_n}, \underline{1})$, 
		\item[-] $D= \underline{-q_n+1} $.  
	\end{itemize}

	Observe that
	\begin{equation*}
	\varphi_n^l=Z_M f^{q_{n}-1}.
	\end{equation*}
	We  claim that
	\begin{enumerate}
		\item\label{h11} The members of the family $\{f^i(T),\; 0\leq i\leq q_n-2\} $ are pairwise disjoint,
		\item\label{h12}  $f^{q_{n}-1}:T\longrightarrow f^{q_{n}-1}(T)$ is diffeomorphism,
		\item \label{h13} $ \dfrac{|f^{q_{n}-1}(M)|}{|f^{q_{n}-1}(S)|}, \dfrac{|f^{q_{n}-1}(M)|}{|f^{q_{n}-1}(D)|}=O\left(\alpha_{n-1}^{\frac{1}{\ell_{1}}, \frac{1}{\ell_{2}}}\right). $
	\end{enumerate}
	
	Points  \textbf{\ref{h11}} and  \textbf{\ref{h12}} come respectively from  \textbf{Proposition~\ref{partition of f}}  and  \textbf{Remark~\ref{diffeo on an interval}}. For point \textbf{\ref{h13}}, observe that  by  \textbf{Proposition~\ref{preimage and gaps adjacent}} and  point \textbf{2} of  \textbf{Fact~\ref{fact}}, 
	\begin{equation*}
	\dfrac{|f^{q_{n}}(M)|}{|f^{q_{n}}(S)|}=	
	\dfrac{|( \underline{q_{n}+1}, \underline{1}|}{|[ \underline{-q_{n-1}+1},\underline{q_n+1})|}<	\dfrac{1}{K}\cdot\dfrac{|( \underline{-q_{n-1}+1}, \underline{1}|}{|[ \underline{-q_{n-1}+1},\underline{1})|} <\dfrac{1}{K}\cdot\alpha_{n-1}.
	\end{equation*}
	Therefore,
	\begin{equation*}
	\dfrac{|f^{q_{n}-1}(M)|}{|f^{q_{n}-1}(S)|}	
	=O(\alpha^{\frac{1}{\ell_{1}}, \frac{1}{\ell_{2}}}_{n-1}).
	\end{equation*}
	Thus, since 
	\begin{equation*}
	\dfrac{|f^{q_{n}-1}(M)|}{|f^{q_{n}-1}(D)|}<\dfrac{|f^{q_{n}-1}(M)|}{|f^{q_{n}-1}(S)|}
	\end{equation*}
	then by  \textbf{Proposition~\ref{koebe principle}}  we get the desired distortion estimate for  $\varphi_n^l$.
	
	\paragraph{Distortion of $\varphi_n$:}\label{distortion phi_n}
	\begin{itemize}
		\item[-] $T=[ \underline{-q_{n-1}+1},\underline{-q_n+(a_{n-1}-1)q_{n-1}+1}],$  
		\item[-] $M=(\underline{q_n+(a_{n-1}-1)q_{n-1}+1}, \underline{1} ),$  
		\item[-] $S=[ \underline{-q_{n-1}+1},\underline{q_n+(a_{n-1}-1)q_{n-1}+1}), $  
		\item[-] $D= (\underline{1},\underline{-q_n+(a_{n-1}-1)q_{n-1}+1}]$. 
	\end{itemize}
	Observe that
\begin{equation*}
\varphi_n=Z_M f^{q_{n-1}-1}.
\end{equation*}
We  claim that
\begin{enumerate}
	\item\label{h111} The members of the family $\{f^i(T),\; 0\leq i\leq q_n-2\} $ are pairwise disjoint,
	\item\label{h122}  $f^{q_{n}-1}:T\longrightarrow f^{q_{n}-1}(T)$ is diffeomorphism,
	\item \label{h133} $ \dfrac{|f^{q_{n-1}-1}(M)|}{|f^{q_{n-1}-1}(S)|}, \dfrac{|f^{q_{n-1}-1}(M)|}{|f^{q_{n-1}-1}(D)|}
	=O\left(\alpha_{n-1}^{\frac{1}{\ell_{1}}, \frac{1}{\ell_{2}}}\right). $
\end{enumerate}

The point  \textbf{\ref{h111}} and  \textbf{\ref{h122}} come respectively of the \textbf{Proposition~\ref{partition of f}}  and the \textbf{Remark~\ref{diffeo on an interval}}. For the point \textbf{\ref{h133}}, observe that  by the \textbf{Proposition~\ref{preimage and gaps adjacent}} and the point \textbf{3} of the \textbf{Fact~\ref{fact}}, 	
\begin{equation*}
\dfrac{|f^{q_{n-1}}(M)|}{|f^{q_{n-1}}(D)|}=	
\dfrac{|( \underline{q_{n+1}+1}, \underline{q_{n-1}+1}|}{|( \underline{q_{n-1}+1},\underline{-q_{n-2}+1})|}<	\dfrac{1}{K}\cdot\dfrac{|( \underline{1}, \underline{-q_{n-2}+1} |}{|(\underline{1},\underline{-q_{n-2}+1}]|} <\dfrac{\alpha_{n-2}}{K}.
\end{equation*}
Therefore, 
\begin{equation*}
\dfrac{|f^{q_{n-1}-1}(M)|}{|f^{q_{n-1}-1}(D)|}= O(\alpha^{\frac{1}{\ell_{1}}, \frac{1}{\ell_{2}}}_{n-2}).
\end{equation*}
Thus, since 
\begin{equation*}
\dfrac{|f^{q_{n-1}-1}(M)|}{|f^{q_{n-1}-1}(S)|}<\dfrac{|f^{q_{n-1}-1}(M)|}{|f^{q_{n-1}-1}(D)|},
\end{equation*}
then by the  \textbf{Proposition~\ref{koebe principle}} we get the desired distortion estimate for  $\varphi_n$.	The Distortion  of $\varphi^r_n$ are obtained by doing similar calculations with:

	\paragraph{Distortion of $\varphi^r_n$: }
	\begin{itemize}
		\item[-] $T=[ \underline{-q_{n}+1},\underline{-q_{n-2}+1}),$  
		\item[-] $M=(\underline{-q_{n}+1},\underline{(a_{n}-1)q_{n}+q_{n-1}+1} ) ,$ 
		\item[-] $S=\underline{-q_{n}+1}$,  
		\item[-] $D=( \underline{(a_{n}-1)q_{n}+q_{n-1}+1},\underline{-q_{n-2}+1}],$  
	\end{itemize}

	Observe that
	\begin{equation*}
	\varphi_n=Z_M f^{q_{n}-1}.
	\end{equation*}
\end{proof}	
\textbf{Proposition~\ref{asymptotic distortion}} shows that $\varphi_{n}^l$, $\varphi_{n}$ and  $\varphi_{n}^r$ go to $Id_{[0,1]} $ when $n$ goes to infinity. As a consequence, the asymptotic behavior of $\RN^nf$ depends only on  $\Sigma^{X} $.

\paragraph{Fibonacci circle maps with a flat piece.}
Let $f\in \mathcal{W}^X$ with rotation number $\rho(f)=[a_0,a_1,\dots]$. Let us recall that for every $n\in\N$, 
$$	\RN^n f:=(x_{1,n} ,x_{2,n} ,x_{3,n} ,x_{4,n} ,s_n,\varphi_n,\varphi^l_n,\varphi^r_1)\in
\mathcal{W}^{X}.$$
By \textbf{(\ref{renormalization definition})}, Figure~\ref{a map 1} and Figure~\ref{a map 2}, it is easy to prove  that
$$x_{2,n} <x_{3,n}\Longleftrightarrow a_n=1.$$

In the rest of the paper we consider $\mathcal{W}^X_{[1]} $ the subclass  of $\mathcal{W}^X$ consisting of  maps  with Fibonacci rotation number.

\subsection{The Asymptotic of Renormalization}\label{Sub Asym Renor}
The set $\Sigma^{X} $ can be redefined as follows:
\paragraph{changes of variables}


\begin{description}
	\item[] $(X)\longrightarrow (S)$. Let
	\begin{equation} \label{change variable S x}\begin{array}{l}
	
	S_{1,n}=\dfrac{x_{3,n}-x_{2,n}}{x_{3,n}},\;S_{2,n}= \dfrac{1-x_{4,n}
	}{1-x_{2,n}},\\\;S_{3,n}=\dfrac{x_{3,n}}{1-x_{4,n}},\;S_{4,n}=-\dfrac{x_{2,n}}{x_{1,n}},\;S_{5,n}=s_n.
	\end{array}
	\end{equation}
	That is,
	\begin{equation}\label{change variable x S}
	\begin{cases}
	\begin{array}{lcl}x_{1,n} &=& \dfrac{S_{3,n}(1-S_{1,n})S_{2,n}}{(1+S_{3,n}(1-S_{1,n})S_{2,n})S_{4,n}} \\
	x_{2,n} &=& \dfrac{S_{3,n}(1-S_{1,n})S_{2,n}}{1+S_{3,n}(1-S_{1,n})S_{2,n}}\\
	x_{3,n} &=& \dfrac{S_{3,n}S_{2,n}}{1+S_{3,n}(1-S_{1,n})S_{2,n}}\\
	x_{4,n}&=& 1-\dfrac{S_{2,n}}{1+S_{3,n}(1-S_1)S_{2,n}}
	\end{array}
	\end{cases}
	\end{equation}

	\item[]$(S) \longrightarrow (Y)$. Let
	\begin{equation}\label{change variable S y}
	y_{1,n}=S_{1,n},\;y_{2,n}=\ln S_{2,n},\;y_{3,n}=\ln S_{3,n},\;y_{4,n}=\ln S_{4,n},\;y_{5,n}=\ln S_{5,n}.
	\end{equation}
\end{description}
\begin{nota} Considering the respective changes of variables $(X)\longrightarrow (S)$ and $(S) \longrightarrow (Y)$, the space $\mathscr{L}^X$ becomes
	successively $\mathscr{L}^{S}$  and $\mathscr{L}^Y$. It will be clear which parametrization of our space	we are using. The space will then be simply denoted by  $\mathscr{L}$ instead of $\mathscr{L}^X$, $\mathscr{L}^{S}$  or $\mathscr{L}^Y$. Similarly, we will denote by 	$\mathcal{W}_{[1]}$ the class of circle maps  with a flat piece and  Fibonacci rotation number.
\end{nota}
\begin{prop}\label{renormalization S}
	Let $f\in\mathcal{W}_{[1]} $. For every $n\in 2\N$,
	\begin{equation*}
	R(S_{1,n},S_{2,n},S_{3,n},S_{4,n},S_{5,n},\varphi_n,\varphi^l_n,\varphi^r_n)
	\end{equation*}
corresponds to
	\begin{equation*}
(S_{1,n+1},S_{2,n+1},S_{3,n+1},S_{4,n+1},S_{5,n+1},\varphi_{n+1},\varphi^l_{n+1},\varphi^r_{n+1})
	\end{equation*}
	where
\begin{equation}
	\begin{array}{rcl}
		S_{1,n+1}&=&1- \dfrac{\left(\varphi_n^l(S_{1,n})\right)^\eg}{\ed(1-\varphi_n^{-1}\circ q^{-1}_{s_n}(1-S_{2,n}))}, \vspace{0.1cm}\\
		S_{2,n+1}  &=&\dfrac{ s^{\ell_2-1}_5\varphi_n^{-1}\circ q^{-1}_{s_n}(S_{1,n} S_{2,n} S_{3,n})}{1-(\varphi_n^l(S_{1,n}))^{\ell_{1}}}, \vspace{0.1cm}\\
		S_{3,n+1} &=&\dfrac{1-\varphi_n^{-1}\circ q^{-1}_{s_n}(1-S_{2,n}))}{\varphi_n^{-1}\circ q^{-1}_{s_n}(S_{1,n} S_{2,n} S_{3,n})},\vspace{0.1cm}\\
		S_{4,n+1} &=&\dfrac{\left[\varphi_n^l(S_{1,n})\right]^\eg}{S_{4,n}},\vspace{0.1cm} \\
		S_{5,n+1}&=& \varphi_n^l(S_{1,n}), \vspace{0.1cm}\\
		\varphi_{n+1}&=&Z_{[S_{1,n},1]}(\varphi_n^l), \vspace{0.1cm} \\
		\varphi^l_{n+1} &=&\varphi_n^r\circ Z_{[\varphi_n^{-1}\circ q^{-1}_{s_n}(1-S_{2,n}),1]}(q_{s_n}\circ\varphi_n), \vspace{0.1cm} \\
		\varphi^r_{n+1} &=&Z_{[0,S_{1,n}]}(\varphi_n^l)\circ Z_{[0,\varphi_n^{-1}\circ q^{-1}_{s_n}(S_{1,n} S_{2,n} S_{3,n})]}(q_{s_n}\circ\varphi_n).
	\end{array}	
\end{equation}	
\end{prop}

\begin{proof}
It follows from equalities in \textbf{(\ref{change variable S x})} and \textbf{Proposition~\ref{first renormalization on x}}.
\end{proof}

\begin{lem}	\label{S1,n goes to zero}
$S_{1,n}=\dfrac{x_{3,n}-x_{2,n}}{x_{3,n}}=O(\alpha_{n+1})$.
\end{lem}
\begin{proof} Observe that,
\begin{equation*}
S_{1,n}=\dfrac{|(\underline{-q_{n+1}+1},\underline{q_{n+1}+1})|}{|(\underline{-q_{n}+1},\underline{1})|}=O\left(\dfrac{|(\underline{1},\underline{q_{n+2}+1})|}{|(\underline{1},\underline{q_{n}+1}]|}\right)=O(\alpha_{n+1})
\end{equation*} 
where we use \textbf{Proposition~\ref{koebe principle}} 
and  \textbf{Proposition~\ref{preimage and gaps adjacent}}.
\end{proof}

Because $\alpha_n$ goes to zero, then by \textbf{Lemma~\ref{S1,n goes to zero}} it follows that  $ y_{1,n}=S_{1,n} $ goes to zero. As consequence, the asymptotic behavior of renormalization ultimately depends on $\;y_{2,n},\;y_{3,n},\;y_{4,n}$ and $y_{5,n}$. Let
\begin{equation*}
w_n(f):=\left(\begin{array}{c} y_{2,n}\\ y_{3,n}\\ y_{4,n}\\
y_{5,n}
\end{array}\right).
\end{equation*}
We have the following result.
\begin{prop}\label{super formular}
	Let $(\eg,\ed)\in(1,2)^2$. Then there exists $\lambda_u>1,\;\lambda_s\in(0,1),\;
	E^u,\;E^s$, $E^+,$  $w_{fix}\in\R^4$ such that the following
	holds. Given $ f\in \mathcal{W}_{[1]} $ with critical exponents  $(\eg,\ed) $, there exists $c_u(f), c'_u(f)<0$, $c_s(f), c'_s(f)$, $c_+(f)$ and $c'_+(f)$ such that for
	all $n:=2p_n\in\N^*$,
	\begin{equation*}
	w_{n}(f)=c_u(f) \lambda_u^{p_n}E^u+c_s(f)
	\lambda_s^{p_n}E^s+c_+(f) E^+ + w_{fix}+O(e,c_u(f),\lambda_u,n)
	\end{equation*}
	and 
	\begin{equation*}
	w_{n+1}(f)=c'_u(f) \lambda_u^{p_n}E^u+c'_s(f)
	\lambda_s^{p_n}E^s+c'_+(f) E^+ + w_{fix}+O(e,c'_u(f),\lambda_u,n)
	\end{equation*}
	where $O(e,c_u(f),\lambda_u,n)$ is the vector whose components are  
	\begin{equation*}
	O\left((e^{c_u(f) \lambda_u^{p_{n-4}}})^{1/\overline{\ell}}\right)
	\end{equation*}
	with
	\begin{equation*}
	\overline{\ell}:=\max\{\ell_1,\ell_2\}.
	\end{equation*}
	Also,
	\begin{equation*}\begin{array}{c}
	dist(\varphi_n)=O\left(e^{\frac{c_u(f) \lambda_u^{p_{n-2}}}{\ell_{1}}}\right),\\
	dist(\varphi_n^l)=O\left(e^{\frac{c_u(f) \lambda_u^{p_{n-1}}}{\ell_{2}}}\right),\\
	dist(\varphi_n^r)=O\left(e^{\frac{c_u(f)
			\lambda_u^{p_{n}}}{\ell_{1}}}\right).
	\end{array}\end{equation*}
\end{prop}

The rest of the section will be devoted largely to prove this proposition. In particular, we are going to show that:

\begin{equation*}
\begin{array}{cl}
1. &\lambda_u=\frac{1}{2\eg\ed}\left(\eg+\ed+1+\sqrt{\eg^2+(2-2\ed)\eg+\ed^2+2\ed+1}
\right),  \\
2. &\lambda_s= \frac{1}{2\eg\ed}\left(\eg+\ed+1-\sqrt{\eg^2+(2-2\ed)\eg+\ed^2+2\ed+1}
\right), \\
3. &  E^+=\left(\begin{array}{cccc}0&0&1&0\end{array}\right)^t, \\
4.  & E^s(\ell_{2},\ell_{1},\lambda_u-\lambda_s)=
\left(\begin{array}{c}e_2^s\\e_3^s(\ell_{2},\ell_{1},\lambda_u-\lambda_s)\\e_4^s(\ell_{2},\ell_{1},\lambda_u-\lambda_s)\\
e_5^s(\ell_{2},\ell_{1},\lambda_u-\lambda_s)\end{array}\right), \\
5. &E^u(\ell_{2},\ell_{1},\lambda_u-\lambda_s)=
\left(\begin{array}{c}e_2^u\\e_3^u(\ell_{2},\ell_{1},\lambda_u-\lambda_s)\\e_4^u(\ell_{2},\ell_{1},\lambda_u-\lambda_s)\\
e_5^u(\ell_{2},\ell_{1},\lambda_u-\lambda_s)\end{array}\right)
\end{array}
\end{equation*}
where
\begin{equation*}
e_2^s= e_2^u=1,
\end{equation*}
\begin{eqnarray*}
	e_3^s(\ell_{2},\ell_{1},-\lambda) &=&
	\frac{(\ell_{2}\ell_{1}-\ell_{2}-1)(\lambda + \ell_{1}-\ell_{2} +1)+ 2\ell_{2}}{(\ell_{1}+\ell_{1}+1)\lambda+(\ell_{1}+1)^{2} +(\ell_{2}+1)^{2}-1 }\\
	&=&	e_3^u(\ell_{2},\ell_{1},\lambda),
\end{eqnarray*}
\begin{eqnarray*}
	e_4^s(\ell_{2},\ell_{1},\lambda) &=&
	\frac{(\ell_{2}^{2}+\ell_{2})\lambda + ( 2 \ell_{1}\ell_{2} +1)(\ell_{2}-1)\ell_{1}\ell_{2}-\ell_{2}(\ell_{2}+1)^{2} }{((\ell_{2}-\ell_{2}^2)\ell_{1}^2+\ell_{1} +(\ell_{1}+1)^{2})\lambda +D}\\
	&=&	e_3^u(\ell_{2},\ell_{1},-\lambda)
\end{eqnarray*}
and 
\begin{equation*}
e_5^s(\ell_{2},\ell_{1},\lambda)=\ell_{2}\lambda_s
=e_5^u(\ell_{2},\ell_{1},-\lambda)
\end{equation*}
with
\begin{equation*}
\lambda= \sqrt{\eg^2+(2-2\ed)\eg+\ed^2+2\ed+1}\end{equation*}
and 
\begin{equation*} D=(\ell_{2}^3-\ell_{2} +(\ell_{2}^2-\ell_{2})\ell_{1}+1)\ell_{1}^2 -(\ell_{2}+2)\ell_{1} -(\ell_{2}+1)^{3}.
\end{equation*}

\subsubsection{ Asymptotic affine behavior of renormalization}
Let $f\in\mathcal{W}_{[1]}$. We are going to show that  $R^{n+1} f =L R^{n}f +C_0(f)$, for every $n$ large enough. We start with the following observation. 
\begin{lem}	\label{asymptotic coordonat in x}
	Let $ f\in\mathcal{W}_{[1]}$. We have:
	\begin{enumerate}
		\item $	\dfrac{x_{2,n+1}}{x_{1,n}}=O(\alpha_{n+1})$,
		\item $\dfrac{x_{3,n+1}}{x_{1,n}}=O(\alpha_{n+1})$,
		\item $\dfrac{x_{1,n}-x_{4,n+1}}{x_{1,n}}=O(\alpha_{n})$,
		\item $s_{n}=S_{5,n}=O(\alpha_{n})$.
	\end{enumerate}
\end{lem}		
\begin{proof}
	The two first points become directly of point \textbf{2} of \textbf{Remark~\ref{nth renormalization}}. Indeed,
	\begin{equation*}
	\dfrac{x_{2,n+1}}{x_{1,n}}=\dfrac{|(\underline{1},\underline{q_{n+2}+1})|}{|(\underline{1},\underline{q_{n}+1})|}<\dfrac{|(\underline{1},\underline{-q_{n+1}+1})|}{|(\underline{1},\underline{-q_{n+1}+1}]|}<\alpha_{n+1}.
	\end{equation*}	
	\begin{equation*}
	\dfrac{x_{3,n+1}}{x_{1,n}}=\dfrac{|(\underline{1},\underline{-q_{n+1}+1})|}{|(\underline{1},\underline{q_{n}+1})|}<\dfrac{|(\underline{1},\underline{-q_{n+1}+1})|}{|(\underline{1},\underline{-q_{n+1}+1}]|}<\alpha_{n+1}.
	\end{equation*}	
	
For point \textbf{3}, observe that  \begin{equation*} (\underline{-q_{n+1}+1},\underline{q_{n}+1}) \subset (\underline{1},\underline{q_{n}+1}). \end{equation*}
Therefore,
	\begin{equation*}
	\dfrac{x_{1,n}-x_{4,n+1}}{x_{1,n}}=\dfrac{|(\underline{-q_{n+1}+1},\underline{q_n+1})|}{|(\underline{1},\underline{q_{n}+1})|}=O\left(\dfrac{|(\underline{-q_{n}+1},\underline{q_{n+1}+1})|}{|(\underline{q_{n-1}+1},\underline{q_{n+1}+1}]|}\right)=O(\alpha_{n})
	\end{equation*}
where we use  \textbf{Proposition~\ref{koebe principle}} and  \textbf{Proposition~\ref{preimage and gaps adjacent}}. 

		By \textbf{Lemma~\ref{S1,n goes to zero}} and  \textbf{Proposition~\ref{asymptotic distortion}}, we get
	\begin{equation*}
	s_{n}=S_{5,n}=S_{1,n-1}\dfrac{\varphi_n^l(S_{1,n-1})}{S_{1,n-1}}=O(\alpha_{n})\left(1+O(\alpha^{\ell_1}_{n-1})\right)=O(\alpha_{n})
	\end{equation*}
	and the lemma is shown.
\end{proof}

\begin{prop}\label{renormazation nth on S}
Let	$ f\in \mathcal{W}_{[1]} $. We have: 
	\begin{equation*}
	\begin{array}{cccc}
	1.& S_{1,n+1}&
	=&1-\frac{\ell_{2}S_{1,n}^{\ell_{1}}}{S_{2,n}}\left(1+O\left(\alpha_{n-1}^{1/\ell_{2}}\right)\right),\vspace{0.15cm}\\
	2.& S_{2,n+1}&
	=&\frac{S_{1,n}S_{2,n} S_{3,n}}{\ell_{2}S_{5,n}^{\ell_{2}-1}}\left(1+O\left(\alpha_{n-2}^{1/\ell_{1}}\right)\right),\vspace{0.15cm}\\
	3.& S_{3,n+1}&
	=&\frac{S_{5,n}^{\ell_{2}-1}}{S_{1,n} S_{3,n}}\left(1+O\left(\alpha_{n-2}^{1/\ell_{1}}\right)\right),\vspace{0.15cm}\\
	4.& S_{4,n+1}&
	=&\frac{S_{1,n}^{\ell_{1}}}{S_{4,n}}\left(1+O\left(\alpha_{n-1}^{1/\ell_{2}}\right)\right),\vspace{0.15cm}\\
	5.& S_{5,n+1}&
	=&S_{1,n}\left(1+O\left(\alpha_{n-1}^{1/\ell_{2}}\right)\right).
	\end{array}
	\end{equation*}
	The case $n$ odd is just a permutation between $\ell_{1} $ and $\ell_{2} $ in the previous equalities, as mentioned in point \textbf{1} of \textbf{Remark~\ref{critical exponent and  renormalization}}. 
\end{prop}
\begin{proof}
	Observe that
	\begin{equation*}
	f^{q_{n-1}}(( \underline{1},\underline{-q_{n+1}+1}))=(\underline{{q_{n-1}+1}},\underline{-q_{n}+1}).
	\end{equation*}
On the other hand, by \textbf{(\ref{explicite renormazation function 0})} we have
	\begin{equation*}
	\RN^nf_{|(x_{1,n}, 0)}=(1-x_{2,n})q_{s_n}\circ\varphi_n\left(\dfrac{x_{1,n}-x}{x_{1,n}}\right) +x_{2,n}.
	\end{equation*}
	Then by the value theorem there is $\theta_1, \theta_2\in(0,1)$ such that, 
	\begin{equation}\label{e41}
	\dfrac{|(\underline{{q_{n-1}+1}},\underline{-q_{n}+1})|}{|( \underline{1},\underline{-q_{n+1}+1})|}=(1-x_{2,n})\varphi'_n(\theta_1)q'_{s_n}(\theta_2).
	\end{equation}
	Remark that $q'_{s_n}(\theta_2):=Dq_{s_n}(\theta_2)=\ell_2O(1)$ and note that by \textbf{Proposition~\ref{asymptotic distortion}}
	$\varphi'_n(\theta_1)=1+O(\alpha^{1/{\ell_1}}_{n-2}).$ Thus, by point \textbf{1} of \textbf{Remark~\ref{nth renormalization}} and equality \textbf{( \ref{e41})}, we obtain
	\begin{equation*}
	S_{2,n}= \dfrac{1-x_{4,n}}{1-x_{2,n}}=\ell_2x_{3,n+1}(1+O(\alpha^{1/{\ell_1}}_{n-2})).
	\end{equation*}
	Therefore,
	\begin{equation}\label{e42}
	x_{3,n+1}	= \dfrac{S_{2,n}}{\ell_2}(1+O(\alpha^{1/{\ell_1}}_{n-2})).
	\end{equation}
	By \textbf{Proposition~\ref{first renormalization on x}} and  \textbf{Proposition~\ref{renormazation nth on S}} we obtain
	\begin{align}
	x_{2,n+1}	= (\varphi^l_n(S_{1,n}))^{\ell_1}= S^{\ell_1}_{1,n} \left(\dfrac{\varphi^l_n(S_{1,n})}{S_{1,n}}\right)^{\ell_1}
	=S^{\ell_1}_{1,n}(1+O(\alpha^{1/{\ell_2}}_{n-1})).\label{e43}
	\end{align} 
	It follows from (\textbf{\ref{e42}}) and (\textbf{\ref{e43}}) that
	\begin{align*}
	S_{1,n+1}	=1- \dfrac{	x_{2,n+1}}{	x_{3,n+1}}=   1- \ell_2\dfrac{S^{\ell_1}_{1,n}}{S_{2,n}}(1+O(\alpha^{1/{\ell_2}}_{n-1}).
	\end{align*}  
	By  \textbf{Proposition~\ref{first renormalization on x}} we have 
	\begin{equation*}
	(1-x_{2,n})q_{s_n}\circ\varphi_n(1-x_{4,n+1} ) =x_{3,n} -x_{2,n}.
	\end{equation*}
	That is,
	\begin{equation*}
	(1-x_{4,n+1})\dfrac{q_{s_n}\circ\varphi_n(1-x_{4,n+1} )}{1-x_{4,n+1}} =\dfrac{x_{3,n} -x_{2,n}}{1-x_{2,n}}.
	\end{equation*}
	Thus,  by point \textbf{3} of \textbf{Lemma~\ref{asymptotic coordonat in x}}  and  \textbf{Proposition~\ref{asymptotic distortion}}
we get
	\begin{align*}
	\ell_2S^{\ell_2-1}_{5,n}(1+O(\alpha^{1/{\ell_1}}_{n-2}))(1-x_{4,n+1} ) &=\dfrac{x_{3,n} -x_{2,n}}{1-x_{2,n}}\\
	&= S_{1,n}S_{2,n}S_{3,n}.
	\end{align*}		
	Therefore,
	\begin{equation}\label{e44}
	1-x_{4,n+1} =\dfrac{S_{1,n}S_{2,n}S_{3,n}}{\ell_2S^{\ell_2-1}_{5,n}}(1+O(\alpha^{1/{\ell_1}}_{n-2})).
	\end{equation}
	By point \textbf{1} of \textbf{Lemma~\ref{asymptotic coordonat in x}} 
	and  equality \textbf{(\ref{e44})}, we have
	\begin{align*}
	S_{2,n+1} =\dfrac{1-x_{4,n+1}}{1-x_{2,n+1}}=\dfrac{S_{1,n}S_{2,n}S_{3,n}}{\ell_2S^{\ell_2-1}_{5,n}}(1+O(\alpha^{1/{\ell_1}}_{n-2})).
	\end{align*}
	From (\textbf{\ref{e42}}) and (\textbf{\ref{e44}}) we obtain
	\begin{align*}
	S_{3,n+1}	=   \dfrac{	x_{3,n+1}}{	 1-x_{4,n+1}} =  \dfrac{S^{\ell_2-1}_{5,n}}{S_{1,n}S_{3,n}}(1+O(\alpha^{1/{\ell_1}}_{n-2})).
	\end{align*} 
	By \textbf{Proposition~\ref{first renormalization on x}} and equality \textbf{(\ref{e43})}, it follows that 
	\begin{align*}
	S_{4,n+1}	=   \dfrac{	x_{2,n+1}}{	 -x_{1,n+1}} 
	=   \dfrac{	x_{2,n+1}}{	 S_{4,n}}=  \dfrac{S^{\ell_1}_{1,n}}{S_{4,n}}(1+O(\alpha^{1/{\ell_2}}_{n-1}))
	\end{align*}
	and finally, by   \textbf{Proposition~\ref{renormalization S}} and \textbf{Proposition~\ref{asymptotic distortion}}, we have
	\begin{align*}
	S_{5,n+1}
	= S_{1,n} \left(\dfrac{\varphi^l_n(S_{1,n})}{S_{1,n}}\right)
	=S_{1,n}(1+O(\alpha^{1/{\ell_2}}_{n-1})).
	\end{align*} 	 
\end{proof}

\begin{cor}\label{corollary renormalization on S}
Let	$ f\in \mathcal{W}_{[1]} $. Then	\begin{equation*}
	\frac{\ell_{2}S_{1,2n}^{\ell_1}}{S_{2,2n}}= 1+O\left(\alpha_{2n-1}^{\frac{1}{ \ell_{2}}}\right)\quad \mbox{  and  } \quad \frac{\ell_{1}S_{1,2n+1}^{\ell_2}}{S_{2,2n+1}}= 1+O\left(\alpha_{2n}^{\frac{1}{ \ell_{1}}}\right).
	\end{equation*}
\end{cor}
\begin{proof}
 By  point \textbf{1} of \textbf{Proposition~\ref{renormazation nth on S}} and  \textbf{Lemma~\ref{S1,n goes to zero}}
	\begin{align*}
	\dfrac{\ell_2 S^{\ell_1}_{1,n}}{S_{2,n}} =	(1-S_{1,{n+1}})(1+O(\alpha^{1/{\ell_2}}_{n-1}))=1+O(\alpha^{1/{\ell_2}}_{n-1}).
	\end{align*} 	
\end{proof}

By \textbf{ Proposition~\ref{renormazation nth on S} } and  \textbf{Corollary~\ref{corollary renormalization on S}}, we have the following result.
\begin{prop}\label{recurrence on Wn}
	Let $ f\in \mathcal{W}_{[1]} $. For $n\in\N$ even, the following equality holds.
	\begin{equation*}
	w_{n+2}=L_{(\ell_{1},\ell_{2})}w_{n+1}+w^*_{(\ell_{1},\ell_{2})}+
	\underline{O}(n,\ell,\alpha)
	\end{equation*}
where
	\begin{equation*}
	L_{(\ell_{1},\ell_{2})}=\left(\begin{array}{cccc}
	1+\frac{1}{\ell_{1}}&1&0&1-\ell_{2}\\-\frac{1}{\ell_{1}}&-1 &0& \ell_{2}-1\\1&0&-1&0\\
	\frac{1}{\ell_{1}}&0 &0& 0
	\end{array}\right),
	\end{equation*}
	\begin{equation*}
	w^*_{(\ell_{1},\ell_{2})}=\left(\begin{array}{c}
	-\left(1+\frac{1}{\ell_{1}}\right) \log\ell_{2}\\\frac{1}{\ell_{1}}\log\ell_{2}\\
	-\log\ell_{2}\\
	-\frac{1}{\ell_{1}}\log\ell_{2}
	\end{array}\right)
	\end{equation*}
	and
	\begin{equation*}
	\underline{O}(n,\ell,\alpha)=\left(\begin{array}{cccc}
	O\left(\alpha_{n-2}^{1/\overline{\ell}}\right),&O\left(\alpha_{n-2}^{1/\overline{\ell}}\right),&
	O\left(\alpha_{n-2}^{1/\overline{\ell}}\right),&
	O\left(\alpha_{n-2}^{1/\overline{\ell}}\right)
	\end{array}\right)^t.
	\end{equation*}
\end{prop}

As noted in the introduction, we study the operator $R^2$ on $(f,\RN f)$ rather than 
$R$ on $f$. Thus,  we define
\begin{equation*}
\overline{L}=\overline{L}_{(\ell_{1},\ell_{2})}=L_{(\ell_{2},\ell_{1})}L_{(\ell_{1},\ell_{2})}
\end{equation*}
and
\begin{equation*}
\overline{w}^*_{(\ell_{1},\ell_{2})}=L_{(\ell_{2},\ell_{1})}w^*_{(\ell_{1},\ell_{2})}+
w^*_{(\ell_{2},\ell_{1})}.
\end{equation*}
We have
\begin{equation}\label{reccurence two pace}
w_{n+2}=\overline{L}_{(\ell_{1},\ell_{2})}w_{n}+\overline{w}^*_{(\ell_{1},\ell_{2})}+
\underline{O}(n,\overline{\ell},\alpha).
\end{equation}

The next step is to explore equality \text{(\ref{reccurence two pace})} in order to conclude the proof of \textbf{Proposition~\ref{super formular}}.

\subsubsection{Asymptotic in Y-Coordinates}
Let $w_{fix} $ be the fixed point of the equation $\overline{L}w+
\overline{w}^*_{(\ell_{1},\ell_{2})} = w$. We have the following result.
\begin{lem}\label{wn en fct de cu,alpha0}
Let	$ f\in \mathcal{W}_{[1]} $. Then,
	\begin{equation*}\label{formular of w_n}
	w_{n}(f)=c_u(f) \lambda_u^{p_n}E^u+c_s(f) \lambda_s^{p_n}E^s+c_+(f)
	E^+ +w_{fix}+O(n,\ell_1,\overline{\ell},\alpha_0);
	\end{equation*}
	\begin{equation*}\label{formular of w_n+1}
	w_{n+1}(f)=c'_u(f) \lambda_u^{p_n}E^u+c'_s(f) \lambda_s^{p_n}E^s+c'_+(f)
	E^+ +w_{fix}+O(n,\ell_2,\overline{\ell},\alpha_0)
	\end{equation*}
	where
	$O(n,\ell_1,\overline{\ell},\alpha_0)$ resp $O(n,\ell_2,\overline{\ell},\alpha_0)$ is the vector whose  components are equal to
	\begin{equation*}	0\left(\alpha_0^{\left(\frac{2}{\ell_1}\right)^{p_{n-4}}}\right)^{\frac{1}{\overline{\ell}}} \mbox{ resp   }\quad 0\left(\alpha_0^{\left(\frac{2}{\ell_2}\right)^{p_{n-4}}}\right)^{\frac{1}{\overline{\ell}}}.
	\end{equation*}
	
\end{lem}

\begin{proof}
	Consider the sequence $(v_n)_{n\in\N}$ defined by
	$v_n=w_n-w_{fix} $.
	From equality \textbf{(\ref{reccurence two pace})} and  \textbf{Proposition~\ref{recurrence formula}}, it follows that,
	\begin{equation*}
	v_{n+2}=\overline{L}v_n+\epsilon_n
	\end{equation*}
	where $\epsilon_n=O(n,\ell_1,\overline{\ell},\alpha_0)$. we obtain
	\begin{equation}\label{vn en fct de rn qn et epsilon}
	v_{n}=\overline{L}^{p_n}v_{0}+\sum_{k=0}^{p_{n-2}}\overline{L}^{p_n-k-1}\epsilon_{2k}.
	\end{equation}
	By expressing $v_{0} $ and $\epsilon_n$ in the eigenbasis, we
	have:
	\begin{equation*}
	v_{0}=c_{u,0}E^u+c_{s,0}E^s+c_{+,0}E^++c_{0,0}E^0,
	\end{equation*}
	\begin{equation*}
	\epsilon_{n}=\epsilon_{u,n}E^u+\epsilon_{s,n}E^s+\epsilon_{+,n}E^++\epsilon_{0,n}E^0.
	\end{equation*}
	We consider the following quantities
	\begin{equation*}
	\begin{array}{cc}
	C_u(f) =& c_{u,0}+\sum_{k=0}^{\infty}\frac{\epsilon_{u,2k}}{\lambda_u^{k+1}}, \\
	C_s(f)  = &c_{s,0} +\sum_{k=0}^{\infty}\frac{\epsilon_{s,2k}}{\lambda_s^{k+1}}, \\
	C_+(f) = & c_{+,0} +\sum_{k=0}^{\infty}{\epsilon_{+,2k}}.\\
	\end{array}
	\end{equation*}
	By introducing $v_ {0} $ and $\epsilon_n$ in
	 \textbf{(\ref{vn en fct de rn qn et epsilon})} we get:
	\begin{align*}
	v_n=&\left(c_{u,0}+\sum_{k=0}^{p_{n-2}}\frac{\epsilon_{u,2k}}{\lambda_u^{k+1}}
	\right)\lambda_u^{p_n}E^u+
	\left(c_{s,0}+\sum_{k=0}^{p_{n-2}}\frac{\epsilon_{s,2k}}{\lambda_s^{k+1}} \right)\lambda_s^{p_n}E^s+\\&+
	\left(c_{+,0}+\sum_{k=0}^{p_{n-2}}{\epsilon_{+,2k}}\right)E^++\epsilon_{0,n-2}E^0\\
	=& c_u(f) \lambda_u^{p_n}E^u+c_s(f) \lambda_s^{p_n}E^s+c_+(f)
	E^++\\
	&+\left(\sum_{k=p_n}^{\infty}\frac{\epsilon_{u,n}}{\lambda_u^{k+1}}
	\right)\lambda_u^{p_n}E^u+
	\left(\sum_{k=p_n}^{\infty}\frac{\epsilon_{s,k}}{\lambda_s^{k+1}} \right)\lambda_s^{p_n}E^s+\\
	&+
	\left(\sum_{k=p_n}^{\infty}{\epsilon_{+,k}}\right)E^++\epsilon_{0,n-2}E^0\\
	=&c_u(f) \lambda_u^{p_n}E^u+c_s(f) \lambda_s^{p_n}E^s+c_+(f)
	E^++O(n-2,\ell_1,\overline{\ell},\alpha_0).
	\end{align*}
	
	In fact, the three sums are estimated by  
	\begin{equation*}	0\left(\alpha_0^{\left(\frac{2}{\ell_1}\right)^{p_{n-2}}}\right)^{\frac{1}{\overline{\ell}}}.
	\end{equation*}
	To convince oneself of that,
	it suffices   to remark that for $k$ large
	enough, the following inequality holds:
	\begin{equation*}
	\left(\frac{\alpha_0^{{\left(\frac{2}{\ell_1}\right)}^{k-1}}}
	{\alpha_0^{{\left(\frac{2}{\ell_1}\right)}^{k-2}}}\right)^{\frac{1}{\overline{\ell}}}<\frac{1}{2}\lambda_u.
	\end{equation*}
	Therefore,
	\begin{align*}
	\left|
	\sum_{k=p_n}^{\infty}\frac{\epsilon_{u,2k}}{\lambda_u^{k+1}}\right|
	\lambda_u^{p_n} = &0\left(\sum_{k=p_n}^{\infty}
	\frac{\left(\alpha_0^{{\left(\frac{2}{\ell_{1}}\right)}^{k-1}}\right)^{\frac{1}{\overline{\ell}}}}
	{\lambda_u^{k}}\lambda_u^{p_{n}-1} \right) \\
	=& 0\left(
	\frac{\left(\alpha_0^{{\left(\frac{2}{\ell_{1}}\right)}^{p_{n-2}}}\right)^{\frac{1}{\overline{\ell}}}}{\lambda_u^{p_n-1}}
	\lambda_u^{p_{n}}
	\sum_{k=p_n}^{\infty} \left(\frac{1}{2}\right)^{k-p_n} \right) \\
	=& 0\left( \left(\alpha_0^{{\left(\frac{2}{\ell_{1}}\right)}^{p_{n-2}}}\right)^{\frac{1}{\overline{\ell}}}\right).\\
	\end{align*}
	And by analogy, we have the other estimates:
\end{proof}
\begin{nota}\label{notation cus+}
	We can write $w_{n}$ (see \textbf{Lemma~\ref{wn en fct de cu,alpha0}}) as follows
	\begin{equation*}
	w_{n}(f)=c^{*}_u(f) \lambda_u^{p_n}E^u+c^{*}_s(f) \lambda_s^{p_n}E^s+c^{*}_+(f)
	E^+ +w_{fix}+O^*(n,\ell_1,\overline{\ell},\alpha_0)
	\end{equation*}
	with $c^{*}_u(f)=c_u(f)$ if $n$ is even and  $c^{*}_u(f)=c'_u(f)$ if $n$ is odd. Also, $c^{\overline{*}}_u(f)=c'_u(f)$ when  $c^{*}_u(f)=c_u(f)$ and vice-versa. We consider the same convention for  $c^{*}_s(f)$ and $c^{*}_+(f)$. $O^*(n,\ell_1,\overline{\ell},\alpha_0)$ is defined from \textbf{Lemma~\ref{wn en fct de cu,alpha0}}.  
\end{nota}
\begin{lem}\label{alpha_n lambda u}
	Let $ f\in \mathcal{W}_{[1]} $. For every $n:=2p_n\in\N$, we have
	\begin{equation*}
	\alpha_{n}=O\left(e^{c_u(f) \lambda_u^{{p_n}}(e_2^u+e_3^u)}\right),\; 	\alpha_{n+1}=O\left(e^{c'_u(f) \lambda_u^{{p_n}}(e_2^u+e_3^u)}\right)\; \mbox{and } c'_u(f),\,c_u(f)<0.\end{equation*}
\end{lem}
\begin{proof}
	We claim that $1-x_{2,n} $ and $x_{4,n} $ are comparable.  Observe that:
	\begin{equation*}
	\dfrac{1-x_{2,n}}{x_{4,n}}=\dfrac{\dfrac{x_{3,n}-x_{2,n}}{x_{4,n}-x_{3,n}}+1+ \dfrac{1-x_{4,n}}{x_{4,n}-x_{3,n}}}{\dfrac{x_{3,n}}{x_{4,n}-x_{3,n}}+1}
	\end{equation*}
	and for $n$ large enough
	\begin{equation*}
	\alpha_n=\dfrac{x_{3,n}}{x_{4,n}},\; \dfrac{x_{3,n}-x_{2,n}}{x_{4,n}-x_{3,n}}
	<\dfrac{x_{3,n}}{x_{4,n}-x_{3,n}}<2\alpha_n.
	\end{equation*}
	Therefore, 
	\begin{equation*}
	\dfrac{1-x_{2,n}}{x_{4,n}}=0(1+\alpha_n).
	\end{equation*}
	On the other hand,
	\begin{equation*}
	\dfrac{1-x_{2,n}}{x_{4,n}}>\dfrac{1}{\dfrac{x_{3,n}}{x_{4,n}-x_{3,n}}+1}>
	\dfrac{1}{2\alpha_n+1}.
	\end{equation*}
	And we have the result. 
	
	Now we have 
	\begin{equation*}
	\alpha_n:=\dfrac{x_{3,n}}{x_{4,n}}=S_{2,n}S_{3,n}\dfrac{1-x_{2,n}}{x_{4,n}}=O(1)S_{2,n}S_{3,n}.
	\end{equation*}
	That is,
	\begin{equation}\label{alpha_n distortion}
	{	\alpha_n}=O({S_{2,n} S_{3,n}}).
	\end{equation}
	By combining equality \textbf{(\ref{alpha_n distortion})} and 
	\textbf{Lemma~\ref{wn en fct de cu,alpha0}})  the lemma  follows.
\end{proof}

\begin{proof}[Proof of \textbf{Proposition~\ref{super formular}}]
The first part of \textbf{Proposition~\ref{super formular}} comes from
\textbf{Proposition~\ref{recurrence formula}}, 
\textbf{Lemma~\ref{wn en fct de cu,alpha0}} and \textbf{Lemma~\ref{alpha_n lambda u}}, and the second part comes from
\textbf{Proposition~\ref{asymptotic distortion}} and \textbf{Lemma~\ref{alpha_n lambda u}}.	
\end{proof}

\begin{cor} $ f\in \mathcal{W}_c $. Let us put $\ell_{2k-1}:=\ell_{1}$ and $\ell_{2k}:=\ell_{2}; k\geq 1$.
	If the sequence
	\begin{equation*}
	\Psi_{0,n} =\dfrac{S^{\ell_{n}}_{1,n}}{S_{2,n}}
	\end{equation*}
	and the logarithm of  sequences 
	\begin{eqnarray*}
		\Psi_{1,n} &=&  \dfrac{S_{2,n}}{\ell_{n+1}(\varphi_{n}^{-1}\circ q^{-1}_{s_{n}}(1-S_{2,n}))}
		\cdot\dfrac{\left(\varphi_{n}^l(S_{1,n})\right)^{\ell_{n}}}{S^{{\ell_{n}}}_{1,n}}, \\
		\Psi_{2,n} &=&
		\dfrac{\ell_{n+1} s^{\ell_n-1}_{5,n}\varphi_{n}^{-1}\circ q^{-1}_{s_{n}}(S_{1,n} S_{2,n} S_{3,n})}{S_{1,n} S_{2,n} S_{3,n}}\cdot \dfrac{1}{1-\varphi_{n}^l(S_{1,n})}, \\
		\Psi_{3,n} &=&\dfrac{S_{1,n}
			S_{3,n}(1-\varphi_{n}^{-1}\circ q^{-1}_{s_{n}}(1-S_{2,n}))}
		{S^{\ell_{n+1}-1}_{5,n}\varphi_{n}^{-1}\circ q^{-1}_{s_{n}}(S_{1,n} S_{2,n} S_{3,n})}, \\
		\Psi_{4,n} &=& \dfrac{\varphi_{n}^l(S_{1,n})}{S_{1,n}},
	\end{eqnarray*}
	are uniformly bounded, then
	
	\begin{enumerate}
		\item $w_n(f)$ tends to infinity exponentially when $(\ell_1, \ell_{2})\in[1,2]^2\setminus\{(2,2)\} $,
		\item $w_n(f)$ is bounded if,  $\lambda_u<1$.
	\end{enumerate}
\end{cor}
\begin{proof}
	From  \textbf{Proposition~\ref{renormalization S}}, we have
\begin{eqnarray*}
	S_{1,1}&=&1-\dfrac{\ed S^{\eg}_1}{S_2}\cdot \left[\dfrac{S_2}{\ed(1-\varphi^{-1}\circ q^{-1}_s(1-S_2))}
	\cdot\dfrac{\left(\varphi^l(S_1)\right)^\eg}{S^{\eg}_1}\right], \\
	S_{2,1}  &=&\dfrac{S_1 S_2 S_3}{\ed s^{\ell_2-1}_5}\cdot
	\left[\dfrac{\ed s^{\ell_2-1}_5\varphi^{-1}\circ q^{-1}_s(S_1 S_2 S_3)}{S_1 S_2 S_3}\cdot \dfrac{1}{1-(\varphi^l(S_1))^{\ell_{1}}}\right], \\
	S_{3,1} &=&\dfrac{S^{\ell_2-1}_5}{S_1 S_3}\cdot\left[\dfrac{S_1 S_3(1-\varphi^{-1}\circ q^{-1}_s(1-S_2))}
	{S^{\ell_2-1}_5\varphi^{-1}\circ q^{-1}_s(S_1 S_2 S_3)}\right],\\
	S_{4,1} &=&\dfrac{S^\eg_1}{S_4}\cdot\left[\left(\dfrac{\varphi^l(S_1)}{S_1}\right)^\eg\right], \\
	S_{5,1}&=& S_1\cdot \left[\dfrac{\varphi^l(S_1)}{S_1}\right], \\
	\varphi&=&Z_{[S_1,1]}(\varphi^l),  \\
	\varphi^l_1 &=&\varphi^r\circ Z_{[\varphi^{-1}\circ q^{-1}_s(1-S_2),1]}(q_s\circ\varphi),  \\
	\varphi^r_1 &=&Z_{[0,S_1]}(\varphi^l)\circ Z_{[0,\varphi^{-1}\circ q^{-1}_s(S_1 S_2 S_3)]}(q_s\circ\varphi).
\end{eqnarray*}		
	 and the result follows by \textbf{Proposition~\ref{super formular}}.
\end{proof}

\subsection{Asymptotic in X-coordinates}
\begin{lem}\label{xn asymptotic coordinate}
	Let $(\ell_1,\ell_{2})\in(1,2)^2$ and $f\in \mathcal{W}_{[1]} $ with critical exponents  $(\ell_1,\ell_{2})$. Then for $n:=2p_n$,
	\begin{equation*}
	\begin{array}{rcl}-x_{1,n} &=&e^{c_u(f)\lambda_u^{p_{n}}(e^u_2+e^u_3-e^u_4) +c_s(f)\lambda_s^{p_{n}}(e^s_2+e^s_3-e^s_4)-c_++ 0((e^{c_u(f)\lambda_u^{p_{n-4}}(e^u_2+e^u_3)})^{1/\overline{\ell}})}\\
	x_{2,n} &=& e^{c_u(f)\lambda_u^{p_{n}}(e^u_2+e^u_3) +c_s(f)\lambda_s^{p_{n}}(e^s_2+e^s_3)+ 0((e^{c_u(f)\lambda_u^{p_{n}}(e^u_2+e^u_3)})^{1/\overline{\ell}})}\\
	x_{3,n} &=& e^{c_u(f)\lambda_u^{p_{n}}(e^u_2+e^u_3) +c_s(f)\lambda_s^{p_{n}}(e^s_2+e^s_3)+ 0((e^{c_u(f)\lambda_u^{p_{n-4}}(e^u_2+e^u_3)})^{1/\overline{\ell}})}\\
	1-x_{4,n}&=& e^{c_u(f)\lambda_u^{p_{n}}e^u_2 +c_s(f)\lambda_s^{p_{n}}e^s_2+ 0((e^{c_u(f)\lambda_u^{p_{n-4}}(e^u_2+e^u_3)})^{1/\overline{\ell}})}
	\end{array}
	\end{equation*}
\end{lem}
\begin{proof}
	
	Recall  that \textbf{(\ref{change variable x S})}
	\begin{equation*}
	\begin{cases}
	\begin{array}{lcl}x_{1,n} &=& \dfrac{S_{3,n}(1-S_{1,n})S_{2,n}}{(1+S_{3,n}(1-S_{1,n})S_{2,n})S_{4,n}} \\
	x_{2,n} &=& \dfrac{S_{3,n}(1-S_{1,n})S_{2,n}}{1+S_{3,n}(1-S_{1,n})S_{2,n}}\\
	x_{3,n} &=& \dfrac{S_{3,n}S_{2,n}}{1+S_{3,n}(1-S_{1,n})S_{2,n}}\\
	x_{4,n}&=& 1-\dfrac{S_{2,n}}{1+S_{3,n}(1-S_1)S_{2,n}}
	\end{array}
	\end{cases}
	\end{equation*}
	\textbf{Proposition~\ref{preimage and gaps adjacent}} implies that 	
	$S_{3,n}S_{2,n}=O(\alpha_n)$ and point 4 of \textbf{Lemma~\ref{asymptotic coordonat in x}} shows that $S_{1,n}=O(\alpha_{n+1})$. Thus, by \textbf{Proposition~\ref{super formular}} the lemma follows.
\end{proof}

\section{Rigidity}\label{Section 5}
 Let $f,g\in\mathcal{W}_{[1]}$ and let $h$ be a conjugacy between $f$ and $g$. Observe that $h(U_f)=U_g$, moreover the choice of $h$ inside $U_f$ could be arbitrary. However, $h|_{K_f} $ is uniquely defined. Being interested in the geometry of $K_f$, we will only study  $h|_{K_f} $ that we will denote $h$ yet. The main question treated in this section is: When do $K_f$ and $K_f$ have the same geometry? In other words, for fixed $f$ in $\mathcal{W}_{[1]}$, what is the geometry class of $f$ ($K_f$)? Before answering this question, we give bi-lipschitz class of $f$. This section is devoted to proving the following result:

\begin{theo}\label{class of h}
	Let $(\ell_{1},\ell_{2})\in(1,2)^2$ and let  $\beta=\dfrac{\overline{c}_u(f)}{\underline{c}_u(f)}\dfrac{(e_2^u+e_3^u)(\lambda_u-1)}{2\overline{\ell}\lambda^2_u}\in(0,1).$
	If $f,g\in\mathcal{W}_{[1]}$ with critical exponents $(\ell_{1},\ell_{2})$ and $h$ is the topological conjugation between $f$ and $g$, then
	\begin{align*}
	h \mbox{ is Holder homeo};\qquad &  \\
	h \mbox{ is a bi-lipschitz homeo}\Longleftrightarrow&  c^*_u(f)=c^*_u(f),  c_+(f)+c'_+(f)=c_+(g)+c'_+(g);\\
	h   \mbox{ is a } C^{1+\beta} \mbox{  diffeo} \Longleftrightarrow&  c^*_u(f)=c^*_u(g), c_+(f)+c'_+(f)=c_+(g)+c'_+(g),\\  & c^*_s(f)=c^*_s(g).
	\end{align*}
\end{theo}
\subsection{Preliminaries}
\begin{prop}\label{h is not Holder}
	Let $(\ell_{1},\ell_{2})\in(1,2)^2$. 
	If $f,g\in\mathcal{W}_{[1]}$ with different unstable eigenvalue. Then  $h$  the topological conjugation between $f$ and $g$ is not Holder.
\end{prop}
\begin{proof}
	Observe that for even $n$,
	\begin{equation*}
	\dfrac{f(U_f)-f^{q_{n+1}+1}(U_f)}{ f^{q_{n}+1}(U_f)-f(U_f) }=\dfrac{x_{2,n}(f)}{- x_{1,n}(f) }=S_{4,n}(f).
	\end{equation*}
	Thus, by \textbf{Proposition~\ref{super formular}}  we obtain,
	\begin{equation}\label{R1}
	f(U_f)-f^{q_{n+1}+1}(U_f)=\prod_{k\leq n}^{}S_{4,k}(f)\sim e^{c_u(f)e^u_4(f)\lambda_u^{p_{n}}(f)+ \frac{n}{2}(c_+(f)+c'_+(f))}.
	\end{equation}
	In the same way
	\begin{equation}\label{R1'}
	g(U_g)-g^{q_{n+1}+1}(U_g)=\prod_{k\leq n}^{}S_{4,k}(g)\sim e^{c_u(g)e^u_4(g)\lambda_u^{p_{n}}(g)+ \frac{n}{2}(c_+(g)+c'_+(g))}.
	\end{equation}
	
	Let us remember that $h(U_f)=U_g$. Therefore, $h(f(U_f))=g(U_g)$.  That is, $h(f^n(U_f))=g^n(U_g)$, for all $n\in\N$. As a consequence, for all $\beta \in (0,1)$
	\begin{equation*}
	\lim\limits_{n\rightarrow \infty}\dfrac{g(U_g)-g^{q_{n+1}+1}(U_g)}{(f(U_f)-f^{q_{n+1}+1}(U_f))^{\beta} }=\begin{cases}
	\begin{array}{lcl}
	0 & \mbox{ si } & \lambda_u(f) <\lambda_u(g)\\
	\infty & \mbox{ si } & \lambda_u(f) >\lambda_u(g)
	\end{array}
	\end{cases}
	\end{equation*}
	and the proposition is shown.
\end{proof}

\subsection{Bi-lipschitz homeomorphism Conjugacy }
\begin{prop}\label{h bi-lipsch, c1 diffeo}
	Let $(\ell_{1},\ell_{2})\in(1,2)^2$. 
	If $f,g\in\mathcal{W}_{[1]}$ with critical exponents $(\ell_{1},\ell_{2})$ and if $h$ is the  topological conjugation between $f$ and $g$ then
	\begin{equation*}
	h \mbox{ is a bi-lipschitz homeo}  \Longrightarrow  c^*_u(f)=c^*_u(g),  c_+(f)+c'_+(f)=c_+(g)+c'_+(g).
	\end{equation*}
\end{prop}				
\begin{proof}
	The proposition comes directly from \textbf{(\ref{R1})} and \textbf{(\ref{R1'})}.
\end{proof}

In the following, we use the below notations. For all $n\in\N$,
\begin{description}
	\item[$\bullet$] $A_n(f)= (f(U_f)  , f^{q_{n}+1}(U_f))$;
	\item[$\bullet$]   $B_n(f)= (f^{-q_{n}+1}(U_f),f(U_f)   )$;
	\item[$\bullet$]  $C_n(f)=f^{-q_{n}+1}(U_f)$;
	\item[$\bullet$]  $D_n(f)= ( f^{q_{n-1}+1}(U_f),f^{-q_{n}+1}(U_f))$. 
\end{description}
And their iterates
\begin{description}
	\item[$\bullet$] $A^{i}_n(f)= f^{i}(A_n(f))$ for $0\leq i< q_{n-1} $;
	\item[$\bullet$]    $B^{i}_n(f)= f^{i}(B_n(f))$ for $0\leq i< q_{n} $;
	\item[$\bullet$]  $C^{i}_n(f)= f^{i}(C_n(f))$ for $0\leq i< q_{n}$;
	\item[$\bullet$]   $D^{i}_n(f)= f^{i}(D_n(f))$ for $0\leq i< q_{n}$. 
\end{description}
Observe that for all $n\in\N$, $$
\mathcal{P}_{n}=\{A^{i}_n(f), B^{j}_n(f), C^{j}_n(f), D^{j}_n(f)|\; 0\leq i< q_{n-1},\;0\leq j< q_{n} \} .$$ Since  $h(U_f)=U_g$, it follows that 

\begin{description}
	\item[$\bullet$] $h(A^{i}_n(f))= A^{i}_n(g)$ for $0\leq i< q_{n-1} $;
	\item[$\bullet$]  $h(B^{i}_n(f))= B^{i}_n(g)$ for $0\leq i< q_{n} $;
	\item[$\bullet$]   $h(C^{i}_n(f))= C^{i}_n(g)$ for $0\leq i< q_{n}$;
	\item[$\bullet$]  $h(D^{i}_n(f))= D^{i}_n(g)$ for $0\leq i< q_{n}$. 
\end{description}

Let $n\in\N$, $Dh_n: [0,1]\longrightarrow \R^+ $ is the function defined by
\begin{equation*}
Dh_n(x)= \dfrac{|h(I)|}{|I|}
\end{equation*}
with $I\in\mathcal{P}_n $ and $x\in  \stackrel{\circ}{I} $. Observe that if  $T=\cup I_i$ with $I_i\in \mathcal{P}_n $, then  for all $m\geq n$, 
\begin{equation}\label{R3}
h(T)= \int_{T}^{}Dh_m.
\end{equation}

\begin{lem}\label{I, minoration}
	There is $C>0$ such that for every interval $I\in \mathcal{P}_n $ the following inequality holds.
	\begin{equation*}
	|I|\geq \frac{1}{C}e^{\underline{c}_u\lambda_u^{p_{n}}e^u_2\frac{2\lambda_u}{\lambda_u-1}}
	\end{equation*}
	with $\underline{c}_u(f)=\min\{c_u(f),c'_u(f)\}.$
\end{lem}
\begin{proof} Suppose that  $I\neq C_n^i$.
	Let $J\in\mathcal{P}_{n-1} $ and $I\subset J$. By the construction of $\mathcal{P}_{n} $ from $\mathcal{P}_{n-1} $, it follows that
	\begin{equation*}
	\dfrac{|I|}{|J|}\geq (1+O(\alpha_{n-2}^{\frac{1}{\ell_{1}}, \frac{1}{\ell_{2}} }))\min\{1-x_{4,n+1}, x_{3,n+1}, S_{1,n}\}
	\end{equation*}
	where we used the mean value theorem and \textbf{Proposition~\ref{asymptotic distortion}}. 
	
	Also, from \textbf{Corollary~\ref{corollary renormalization on S}}, \textbf{Lemma~\ref{xn asymptotic coordinate}}  and \textbf{Proposition~\ref{super formular}} we get:
	\begin{description}
		\item[] $S_{1,n}\geq e^{\frac{c*_u(f)}{\underline{\ell}}\lambda_u^{p_{n}}e^u_2(f) -K\lambda_s^{p_{n}}};$
		\item[] $x_{3,n+1}\geq e^{c^{\underline{*}}_u(f)\lambda_u^{p_{n+1}}(e^u_2(f)+e^u_3(f)) -K\lambda_s^{p_{n}}};$
		\item[]	$1-x_{4,n+1}\geq e^{c^{\underline{*}}_u(f)\lambda_u^{p_{n+1}}e^u_2(f) -K\lambda_s^{p_{n}}}.$
	\end{description}
	
	Thus, from points \textbf{ 1}, \textbf{2}, \textbf{4} and \textbf{5} of \textbf{Proposition~\ref{super formular}}, we obtain
	\begin{equation*}
	\dfrac{|I|}{|J|}\geq e^{\underline{c}_u(f)e^u_2(f)\lambda_u^{p_{n+1}} -K\lambda_s^{p_{n}}}.
	\end{equation*}
	Therefore, if  $J\neq C_{n-1}^j$, we can repeat these estimates  and we get 
	\begin{equation*}
	|I|\geq e^{\underline{c}_u(f)e_2^u(f)\sum\limits_{i=3}^{n+1}\lambda_u^{p_i}-K\sum\limits_{i=1}^{n}\lambda_s^{p_i} }\geq e^{2\underline{c}_u(f)e_2^u(f)\sum\limits_{i=1}^{p_{n+1}}\lambda_u^{p_{2i}}-K\sum\limits_{i=1}^{n}\lambda_s^{p_i} }.
	\end{equation*}
	Moreover,  by \textbf{Proposition~\ref{preimage and gaps adjacent}} we have $ C_{n}^j> K' B^j_{n}$ for some $K'>0$. And the lemma follows.
\end{proof}
\begin{lem}\label{I, majoration}
	There is $C>0$ such that for every interval $I\in \mathcal{P} $ the following inequality holds.
	\begin{equation*}
	|I|\leq \dfrac{1}{C}e^{\overline{c}_u(f)(e_2^u(f)+e_3^u(f))\lambda_u^{p_n}}
	\end{equation*}
	with $\overline{c}_u(f)=\max\{c_u(f),c'_u(f)\}.$
\end{lem}
\begin{proof}
	By mean value theorem and \textbf{Lemma~\ref{asymptotic distortion}} we get 
	\begin{description}
		\item[1.] $|A^{i}_n|= x_{2,n-1}(1+O(\alpha_{n-2}^{\frac{1}{\ell_{1}}, \frac{1}{\ell_{2}}}))$;
		\item[2.]  $|B^{i}_n|=x_{3,n}(1+O(\alpha_{n-2}^{\frac{1}{\ell_{1}}, \frac{1}{\ell_{2}}}))$;
		\item[2.]  $|D^{i}_n|=(1-x_{4,n})(1+O(\alpha_{n-2}^{\frac{1}{\ell_{1}}, \frac{1}{\ell_{2}}}))$. 
	\end{description}
	
	And by  \textbf{Lemma~\ref{xn asymptotic coordinate}}  there exists $K>0$ such that
	
	\begin{description}
		\item[1.] $ x_{2,n-1}\leq Ke^{{c^*_u}(f)e_2^u(f)\lambda_u^{p_{n-1}}} $;
		\item[2.]  $x_{3,n}\leq Ke^{{c^{\overline{*}}_u}(f)(e_2^u(f)+e_2^u(f))\lambda_u^{p_{n}}}$;
		\item[3.]  $(1-x_{4,n})\leq Ke^{{c^{\overline{*}}_u}(f)e_2^u(f)\lambda_u^{p_{n}}} $. 
	\end{description}
	Since $ \lambda_u e_2^u(f)>e_2^u(f)>e_2^u(f)+e_3^u(f)$, we have the lemma.
\end{proof}

\begin{lem}\label{Dhn cauchy}
	Let $(\ell_{1},\ell_{2})\in(1,2)^2$. If $f,g\in\mathcal{W}_{[1]}$ with critical exponent
	$(\ell_{1},\ell_{2})$ such  $c^*_u(f)=c^*_u(g)$ and $c_+(f)+c'_+(f)=c_+(g)+c'_+(g)$ then 
	\begin{equation*}
	\log\dfrac{Dh_{n+1}}{Dh_{n}}=O\left(\alpha_{n-2}^{\frac{1}{\ell_{1}},\frac{1}{\ell_{2}}}+|c^*_s(g)-c^*_s(f)|\lambda_s^{p_n}\right)
	\end{equation*}
\end{lem}
\begin{proof}
	From  \textbf{Lemma~\ref{alpha_n lambda u}} it follows that, if $c^*_u(f)=c^*_u(g)$,  then $\alpha_n(f)=\alpha_n(g)$. Thus, by 
	 \textbf{Lemma~\ref{asymptotic distortion}} and  \textbf{Lemma~\ref{asymptotic coordonat in x}} we get 
	\begin{align*}
	\dfrac{{Dh_{n+1}}_{|A^{i}_{n+1}(f)}}{{Dh_{n}}_{|A^{i}_{n+1}(f)}} = \dfrac{\dfrac{|A^{i}_{n+1}(g)|}{|A^{i}_{n+1}(f)|}}{\dfrac{|B^{i}_{n}(g)|}{|B^{i}_{n}(f)|}}&=\dfrac{\dfrac{|A^{i}_{n+1}(g)|}{|B^{i}_{n}(g)|}}{\dfrac{|A^{i}_{n+1}(f)|}{|B^{i}_{n}(f)|}}\vspace{0.2cm}\\
	&= \dfrac{\dfrac{|A_{n+1}(g)|}{|B_{n}(g)|}}{\dfrac{|A_{n+1}(f)|}{|B_{n}(f)|}}(1+O(\alpha_{n-2}^{\frac{1}{\ell_1}, \frac{1}{\ell_2}}))\vspace{0.2cm}\\
	&= \dfrac{\dfrac{x_{2,n}(g)}{x_{3,n}(g)}}{\dfrac{x_{2,n}(f)}{x_{3,n}(f)}}(1+O(\alpha_{n-2}^{\frac{1}{\ell_1}, \frac{1}{\ell_2}}))\vspace{0.2cm}\\
	&= \dfrac{1-S_{1,n}(g)}{1-S_{1,n}(f)}(1+O(\alpha_{n-2}^{\frac{1}{\ell_{1}}, \frac{1}{\ell_2}}))\vspace{0.2cm}\\
	&= (1+O(\alpha_{n-2}^{\frac{1}{\ell_{1}}, \frac{1}{\ell_{2}}}))
	\end{align*}
	
	Observe that for $i\geq q_{n-1} $, $B^{i}_{n+1}(f)=D^{i-q_{n-1}}_{n+1}(f)$. Therefore, 
	\begin{align*}
	\dfrac{{Dh_{n+1}}_{|B^{i}_{n+1}(f)}}{{Dh_{n}}_{|B^{i}_{n+1}(f)}} &=1.
	\end{align*}
	For $i<q_{n-1} $, with a similar calculation as before, we have
	\begin{align*}
	\dfrac{{Dh_{n+1}}_{|B^{i}_{n+1}(f)}}{{Dh_{n}}_{|B^{i}_{n+1}(f)}}
	&= \dfrac{\dfrac{|B_{n+1}(g)|}{|A_{n}(g)|}}{\dfrac{|B_{n+1}(f)|}{|A_{n}(f)|}}(1+O(\alpha_{n-2}^{\frac{1}{\ell_1}, \frac{1}{\ell_2}}))\vspace{0.2cm}\\
	&= \dfrac{x_{3,n+1}(g)}{x_{3,n+1}(f)}(1+O(\alpha_{n-2}^{\frac{1}{\ell_1}, \frac{1}{\ell_2}}))\vspace{0.2cm}\\
	&= (1+O(\alpha_{n-2}^{\frac{1}{\ell_{1}}, \frac{1}{\ell_2}}))e^{(c^*_s(g)-c^*_s(g)) (e_2^s+e_3^s)\lambda_s^{p_n} }\vspace{0.2cm}\\
	&= 1+O(\alpha_{n-2}^{\frac{1}{\ell_{1}},\frac{1}{\ell_{2}}}+|c^*_s(g)-c^*_s(f)|\lambda_s^{p_n})
	\end{align*} 
	where we used \textbf{Lemma~\ref{xn asymptotic coordinate}}.
	
	Let us remark also that for  $i\geq q_{n-1} $, $C^{i}_{n+1}(f)=C^{i-q_{n-1}}_{n+1}(f)$. Thus 
	\begin{align*}
	\dfrac{{Dh_{n+1}}_{|C^{i}_{n+1}(f)}}{{Dh_{n}}_{|C^{i}_{n+1}(f)}} &=1
	\end{align*}
	Now, when $i<q_{n-1} $, with a similar calculation as before we obtain
	\begin{align*}
	\dfrac{{Dh_{n+1}}_{|C^{i}_{n+1}(f)}}{{Dh_{n}}_{|C^{i}_{n+1}(f)}}
	&= \dfrac{\dfrac{|C_{n+1}(g)|}{|A_{n}(g)|}}{\dfrac{|C_{n+1}(f)|}{|A_{n}(f)|}}(1+O(\alpha_{n-2}^{\frac{1}{\ell_1}, \frac{1}{\ell_2}}))\vspace{0.2cm}\\
	&=\dfrac{ x_{4,n+1}(g)-x_{3,n+1}(g)}{x_{4,n+1}(f)-x_{3,n+1}(f)}(1+O(\alpha_{n-2}^{\frac{1}{\ell_1}, \frac{1}{\ell_2}}))\vspace{0.2cm}\\
	&=\dfrac{1- [x_{3,n+1}(g)+(1-x_{4,n+1}(g))]}{1- [x_{3,n+1}(f)+(1-x_{4,n+1}(f))]}(1+O(\alpha_{n-2}^{\frac{1}{\ell_1}, \frac{1}{\ell_2}}))\vspace{0.2cm}\\
	&=\dfrac{1-\alpha_{n+1}(g)}{1-\alpha_{n+1}(f)}(1+O(\alpha_{n-2}^{\frac{1}{\ell_1}, \frac{1}{\ell_2}}))\vspace{0.2cm}\\
	&= 1+O(\alpha_{n-2}^{\frac{1}{\ell_{1}},\frac{1}{\ell_{2}}})
	\end{align*}
	where we used \textbf{Lemma~\ref{asymptotic coordonat in x}} and \textbf{Lemma~\ref{alpha_n lambda u}}.
	
	Let $\Delta_n(f)$ be defined by $\Delta_n(f)=(f^{-q_{n}+1}(U), f^{q_{n+1}+1}(U))$ and its iterates given by $\Delta^i_n(f)=f^i(\Delta_n(f))$, $i<q_{n} $. Observe that for  $i\geq q_{n-1} $, $D^{i}_{n+1}(f)=\Delta_n^{i-q_{n-1}}(f)$. Thus, since $S_{1,n}=O(\alpha_{n+1})$ (see point \textbf{4} of \textbf{Lemma~\ref{asymptotic coordonat in x}}), it follows that
	\begin{align*}
	\dfrac{{Dh_{n+1}}_{|D^{i}_{n+1}(f)}}{{Dh_{n}}_{|D^{i}_{n+1}(f)}}
	&= \dfrac{\dfrac{|\Delta_n(g)|}{|B_n(g)|}}{\dfrac{|\Delta_n(f)|}{|B_n(f)|}}(1+O(\alpha_{n-2}^{\frac{1}{\ell_1}, \frac{1}{\ell_2}}))\vspace{0.2cm}\\
	&=\dfrac{S_{1,n+1}(g)}{S_{1,n+1}(f)}(1+O(\alpha_{n-2}^{\frac{1}{\ell_1}, \frac{1}{\ell_2}}))\vspace{0.2cm}\\
	&= 1+O(\alpha_{n-2}^{\frac{1}{\ell_{1}},\frac{1}{\ell_{2}}}).
	\end{align*}
	For $i<q_{n-1} $, we obtain
	\begin{align*}
	\dfrac{{Dh_{n+1}}_{|D^{i}_{n+1}(f)}}{{Dh_{n}}_{|D^{i}_{n+1}(f)}}
	&= \dfrac{\dfrac{|D_{n+1}(g)|}{|A_n(g)|}}{\dfrac{|D_{n+1}(g)(f)|}{|A_n(f)|}}(1+O(\alpha_{n-2}^{\frac{1}{\ell_1}, \frac{1}{\ell_2}}))\vspace{0.2cm}\\
	&=\dfrac{1-x_{4,n+1}(g)}{1-x_{4,n+1}(f)}(1+O(\alpha_{n-2}^{\frac{1}{\ell_1}, \frac{1}{\ell_2}}))\vspace{0.2cm}\\
	&= (1+O(\alpha_{n-2}^{\frac{1}{\ell_{1}}, \frac{1}{\ell_2}}))e^{(c^*_s(g)-c^*_s(g)) e_2^s\lambda_s^{p_n} }\vspace{0.2cm}\\
	&= 1+O(\alpha_{n-2}^{\frac{1}{\ell_{1}},\frac{1}{\ell_{2}}}+|c^*_s(g)-c^*_s(f)|\lambda_s^{p_n}).
	\end{align*}
	This ends the proof of the lemma.
\end{proof}

Note that, every boundary point of an interval in $\mathcal{P}_n$ is in the orbit of the critical points $o(f(U))$. So, if $x\not\in o(f(U) $ then $Dh_n$ is well defined. \textbf{Lemma~\ref{Dhn cauchy}} implies that
\begin{equation*}
D(x)=\lim\limits_{n\rightarrow\infty}Dh_n(x)
\end{equation*}
exists. Moreover,  \begin{equation}\label{R00}0<\inf_xD(x)<\sup_xD(x)<\infty.\end{equation}

\begin{lem}\label{Dhn go to D}
	Let $(\ell_{1},\ell_{2})\in(1,2)^2$. If $f,g\in\mathcal{W}_{[1]}$ with critical exponents
	$(\ell_{1},\ell_{2})$ and $c^*_u(f)=c^*_u(g)$, then 
	\begin{equation*}
	D: [0,1]\setminus o(f(U))\longrightarrow (0,\infty)
	\end{equation*}
	is continuous. In particular, for all $n\in\N$,
	\begin{equation*}
	\log\dfrac{D(x)}{Dh_n(x)}=O(\alpha_{n-2}^{\frac{1}{\ell_{1}},\frac{1}{\ell_{2}}}+|c^*_s(g)-c^*_s(f)|\lambda_s^{p_n}). 
	\end{equation*}
\end{lem}
\begin{proof}
	Fix $n_0\in\N$, then by \textbf{Lemma~\ref{Dhn cauchy}} we have 
	\begin{align*}
	\left|\log\dfrac{D(x)}{Dh_{n_0}(x)}\right| &\leq \sum\limits_{k\geq 0}^{}\log\dfrac{Dh_{n_0+k+1}(x)}{Dh_{n_0+k}(x)}\\
	&\leq \sum\limits_{k\geq 0}^{} O(\alpha_{k-2}^{\frac{1}{\ell_{1}},\frac{1}{\ell_{2}}}+|c^*_s(g)-c^*_s(f)|\lambda_s^{p_k})\\ &\leq
	O(\alpha_{n_0-2}^{\frac{1}{\ell_{1}},\frac{1}{\ell_{2}}}+|c^*_s(g)-c^*_s(f)|\lambda_s^{p_{n_0}}).
	\end{align*}
	Therefore, if $x,y\in I\in\mathcal{P}_{n_0} $, we have
	\begin{align*}
	\left|\log\dfrac{D(x)}{D(y)}\right| &\leq \left|\log\dfrac{D(x)}{Dh_{n_0}(x)}\right| +\left| \log\dfrac{D(y)}{Dh_{n_0}(x)}\right| \\
	&\leq
	O(\alpha_{n_0-2}^{\frac{1}{\ell_{1}},\frac{1}{\ell_{2}}}+|c^*_s(g)-c^*_s(f)|\lambda_s^{p_{n_0}})
	\end{align*}
	which means that $D$ is continuous. 
\end{proof}

Let $T$ be an interval in $[0,1]$, then by formula \textbf{(\ref{R3})},
\begin{equation}
|h(T)|=\lim\limits_{n\rightarrow \infty}|h_n(T)|=\lim\limits_{n\rightarrow \infty}\int_{T}Dh_n\leq \int_T\sup_xD(x)\leq \sup_xD(x)|T|.\label{R4}
\end{equation}
Also,
\begin{equation}
|h(T)|=\lim\limits_{n\rightarrow \infty}|h_n(T)|=\lim\limits_{n\rightarrow \infty}\int_{T}Dh_n\geq \int_T\inf_xD(x)\geq \inf_xD(x)|T|.\label{R4'}
\end{equation}

\begin{prop}\label{h homeo diffeo}
	Let $(\ell_{1},\ell_{2})\in(1,2)^2$. Let $f,g\in\mathcal{W}_{[1]}$ with critical exponents
	$(\ell_{1},\ell_{2})$. Then, 
	\begin{equation*}
	h \mbox{ is a bi-lipschitz homeo} \Longleftarrow  c^*_u(f)=c^*_u(g), c_+(f)+c'_+(f)=c_+(g)+c'_+(g).
	\end{equation*}
\end{prop}
\begin{proof} If $c^*_u(f)=c^*_u(g)$  and $c_+(f)+c'_+(f)=c_+(g)+c'_+(g)$ then by  inequalities \textbf{(\ref{R4})}, \textbf{(\ref{R4'})},  \textbf{Lemma~\ref{Dhn go to D}} and   inequalities in \textbf{(\ref{R00})}, $h$  is a bi-Lipschitz.
\end{proof}

As consequence of \textbf{Proposition~\ref{h homeo diffeo}}, we have:
\begin{cor}
Fix $(\ell_{1},\ell_{2})\in(1,2)^2$.
Let $f,g\in\mathcal{W}_{[1]}$ with critical exponents $(\ell_{1},\ell_{2})$ and let $h$ be the topological conjugation between $f$ and $g$. If
	$c^*_u(f)=c^*_u(f)$  and $c_+(f)+c'_+(f)=c_+(g)+c'_+(g)$, then $K_f$ and $K_g$ have the same Hausdorff dimension\footnote{\begin{equation*}
		Dim_HK_f:=\inf_{s\geq0}\left\{H_s(K_f):=\inf_{\delta>0}\left\{\sum |U_i|^{s},\; K_f\subset\cup U_i\,|\;|U_i|<\delta\right\}=0\right\}.
		\end{equation*}}.
\end{cor}
\subsection{ $C^{1+\beta} $  diffeomorphism Conjugacy}
The rigidity class is described as:
\begin{prop}\label{h c1+beta diffeo}
	Let $(\ell_{1},\ell_{2})\in(1,2)^2$ and let 
	$\beta= \dfrac{\underline{c}_u(f)}{\overline{c}_u(f)}\dfrac{(e_2^u+e_3^u)(\lambda_u-1)}{2\overline{\ell}\lambda_u}\in(0,1).$ 
	If $f,g\in\mathcal{W}_{[1]}$ with critical exponents $(\ell_{1},\ell_{2})$ and $h$ is the topological conjugation between $f$ and $g$, then
	
	\begin{align*}
	h   \mbox{ is a } C^{1+\beta} \mbox{  diffeo}  \Longleftrightarrow&  c^*_u(f)=c^*_u(g), c_+(f)+c'_+(f)=c_+(g)+c'_+(g)\\& \mbox{ and }  c^*_s(f)=c^*_s(g).
	\end{align*}
\end{prop}
\begin{proof}
	Suppose that $	h   \mbox{ is a } C^{1+\beta} \mbox{  diffeo} $, then by \textbf{Proposition~\ref{h bi-lipsch, c1 diffeo}}, $c^*_u(f)=c^*_u(g)$ and $c_+(f)+c'_+(f)=c_+(g)+c'_+(g)$. Thus, by \textbf{Proposition~\ref{super formular}} and \textbf{Lemma~\ref{alpha_n lambda u}}
	we get
	\begin{align*}
	\dfrac{Dh(f^{q_{n+1}+1}(U_f))}{Dh(f^{q_{n}+1}(U_f))} &=\dfrac{Dg^{q_{n-1}}(g^{q_{n}}(0))}{Df^{q_{n-1}}(f^{q_{n}}(0))}\vspace{0.2cm}\\
	&=\dfrac{Dq_{s_n(g)(0)}\dfrac{1-x_{2,n}(g)}{x_{1,n}(g)}}{Dq_{s_n(g)(0)}\dfrac{1-x_{2,n}(g)}{x_{1,n}(g)}}(1+O(\alpha_{n-2}^{\frac{1}{\ell_{1}},\frac{1}{\ell_{2}}}))\vspace{0.2cm}\\
	&=\dfrac{s^{\frac{1}{\ell_{1}},\frac{1}{\ell_{2}}}_n(g)(1+O(s_n(g))) \dfrac{1-x_{2,n}(g)}{x_{1,n}(g)}}{s^{\frac{1}{\ell_{1}},\frac{1}{\ell_{2}}}_n(f)(1+O(s_n(f)))\dfrac{1-x_{2,n}(g)}{x_{1,n}(g)}}(1+O(\alpha_{n-2}^{\frac{1}{\ell_{1}},\frac{1}{\ell_{2}}}))\vspace{0.2cm}\\
	&=e^{(c^*_s(f)-c^*_s(g))\lambda_s^{p_{n}}(e^s_2+e^s_3-e^s_4)+ 0((e^{c^*_u(f)\lambda_u^{p_{n-4}}(e^u_2+e^u_3)})^{1/\overline{\ell}})}
	\end{align*}
	Let us note that, by point \textbf{5} of \textbf{Proposition~\ref{super formular}}
	$e^s_2+e^s_3-e^s_4\neq 0$.
	Moreover, because $h   \mbox{ is a } C^{1+\beta} \mbox{  diffeo} $, then there exists $K>0$ such that
	\begin{equation*}
	\dfrac{Dh(f^{q_{n+1}+1}(U_f))}{Dh(f^{q_{n}+1}(U_f))}\leq K|(f^{q_{n+1}+1}(U_f),f^{q_{n}+1}(U_f))|^{\beta}\leq \dfrac{2K}{C}e^{\beta\overline{c}_u(f)(e_2^u+e_3^u)\lambda_u^{p_n}} 
	\end{equation*}
	where we used \textbf{Lemma~\ref{I, majoration}}. 
	The two above estimates imply that $c^*_s(f)=c^*_s(g)$.
	
	Let us suppose that  $c^*_s(f)=c^*_s(g)$, $c^*_u(f)=c^*_u(g)$ and  $c_+(f)+c'_+(f)=c_+(g)+c'_+(g)$.
	
	We are going to  show first that, under conditions $c^*_u(f)=c^*_u(g)$ and  $c_+(f)+c'_+(f)=c_+(g)+c'_+(g)$,  $h$ is $C^1$ diffeo.
	
	Since by \textbf{(\ref{R4})} $h$ is differentiable with $Dh(x)=D(x)$, for all $x$, then it remains to prove that $D$ can be extended to a continuous function. This is possible, if only if, for all $k\geq 0$ 
	\begin{equation*}
	\lim\limits_{x\rightarrow c_k^-}D(x)=\lim\limits_{x\rightarrow c_k^+}D(x)
	\end{equation*}
	where $c_k=f^k(f(U))$. 
	
	Let us fix $k\geq 0$ and let $n\in\N$ big enough such that $q_{n-1}>k$. Observe that
	
	\begin{equation*}
	D_-(c_k):=\lim\limits_{x\rightarrow c_k^-}D(x)=\lim\limits_{n\rightarrow \infty}\dfrac{|B^k_{2n}(g)|}{B^k_{2n}(f)|}
	\end{equation*}
	and 
	\begin{equation*}
	D_+(c_k):=\lim\limits_{x\rightarrow c_k^+}D(x)=\lim\limits_{n\rightarrow \infty}\dfrac{|A^k_{2n}(g)|}{A^k_{2n}(f)|}
	\end{equation*}
	By \textbf{Proposition~\ref{asymptotic distortion}} and  \textbf{Lemma~\ref{xn asymptotic coordinate}}, we get
	\begin{align*}
	\dfrac{D_-(c_k)}{D_+(c_k)}=\lim\limits_{n\rightarrow \infty} \dfrac{\dfrac{|B^{k}_{2n}(g)|}{|A^{k}_{2n}(g)|}}{\dfrac{|B^{k}_{2n}(f)|}{|A^{k}_{2n}(f)|}}
	=\lim\limits_{n\rightarrow \infty} \dfrac{\dfrac{|B_{2n}(g)|}{|A_{2n}(g)|}}{\dfrac{|B_{2n}(f)|}{|A_{2n}(f)|}}\vspace{0.2cm}
	=\lim\limits_{n\rightarrow \infty} \dfrac{\dfrac{x_{3,2n}(g)}{-x_{1,2n}(g)}}{\dfrac{x_{3,2n}(f)}{-x_{1,2n}(f)}}=\lim\limits_{n\rightarrow \infty}e^{\lambda_s^n}=1
	\end{align*}
	
	Now, let $x,y\in K_f$ and  choose $n$ maximal such that there exists $I\in\mathcal{P}_n $ and $I$ contains $[x,y]$. Since 
	$B_{n+1}^{i+q_{n-1}}=D_n^i$, then by maximality of $n$, $I\neq D_n^i$. Also, as $x,y\in K_f$, then $I\neq C_n^i$. So, either $I=A_n^i$ or  $I=B_n^i$. 
	In the case where $I=A_n^i$, then we can assume that  $x\in D^{i}_{n+1} $ and $y\in B^{i}_{n+1} $. Thus,  by \textbf{Lemma~\ref{cn and the other side}} and by \textbf{Lemma~\ref{I, minoration}} we obtain
	\begin{equation}\label{R5}
	|x-y|\geq K| A_n^i|\geq \frac{K}{C}e^{\underline{c}_u\lambda_u^{p_{n}}e^u_2\frac{2\lambda_u}{\lambda_u-1}}.
	\end{equation}
	Observe that by \textbf{Lemma~\ref{Dhn go to D}} and \textbf{Lemma~\ref{alpha_n lambda u}} we have 
	\begin{equation}\label{R6}
	\left|\log\dfrac{ Dh(x)}{ Dh(y)}\right|=O(e^{\frac{1}{\overline{\ell}}c^*_u\lambda_u^{p_{n-2}}(e^u_2+e^u_3)})
	=O(e^{\frac{1}{{\overline{\ell}}}\overline{c}_u\lambda_u^{p_{n-2}}(e^u_2+e^u_3)}).
	\end{equation}
	By (\textbf{\ref{R5}}) and (\textbf{\ref{R6}}), we have
	\begin{equation*}
	\dfrac{\left| Dh(x)-Dh(y)\right|}{| x-y|^{\beta}}=O\left(e^{\lambda_u^{p_{n-2}}\left[\frac{1}{\overline{\ell}}\overline{c}_u(e^u_2+e^u_3)-\beta\underline{c}_ue^u_2\frac{2\lambda^2_u}{\lambda_u-1}\right]}\right)
	=O(1)
	\end{equation*}
	for $0<\beta<\dfrac{\overline{c}_u(e^u_2+e^u_3)(\lambda_u-1)}{2\underline{c}_u\overline{\ell}e^u_2 \lambda^2_u}=\dfrac{\overline{c}_u(1+e^u_3)(\lambda_u-1)}{2\underline{c}_u\overline{\ell} \lambda^2_u} $, see point \textbf{5} of  \textbf{Proposition~\ref{super formular}}.
	
	For $I=B_n^i$, then we can choose $(x,y)$ such that $x\in A^{i}_{n+1} $  and 
	$y\in B^{i+q_{n-1}}_{n+1} $.
	
	If $x\in(f^{q_{n+3}+i+1}(U_f), f^{i+1}(U_f)) $, then for  $m=n+1$, \textbf{Lemma~\ref{cn and the other side}} and by \textbf{Lemma~\ref{I, minoration}} imply that
	\begin{equation}\label{R7}
	|x-y|\geq K| A_{m}^i|\geq \frac{K}{C}e^{\underline{c}_u\lambda_u^{p_{m}}e^u_2\frac{2\lambda_u}{\lambda_u-1}}.
	\end{equation}
	
	Otherwise, let $k>0$, $m=n+2k $  maximal such $x\in A^{i+q_{n+1}}_{m} $ and $y\in B^{i+q_{n+1}}_{m} $. By \textbf{Lemma~\ref{I, minoration}}, we get
	\begin{equation}\label{R8}
	|x-y|\geq\frac{1}{2} \min\{ | A^{i+q_{n+1}}_{m}|, |  B^{i+q_{n+1}}_{m}|\} \geq\frac{1}{C}e^{\underline{c}_u\lambda_u^{p_{m}}e^u_2\frac{2\lambda_u}{\lambda_u-1}}.
	\end{equation}
	
On the other hand, for all $x,y\in B_n^i$, 
 	\begin{equation}\label{R9}
	\left|\log\dfrac{ Dh(x)}{ Dh(y)}\right|=O(e^{\frac{1}{\overline{\ell}}c^*_u\lambda_u^{p_{m-2}}(e^u_2+e^u_3)})
	=O(e^{\frac{1}{{\overline{\ell}}}\overline{c}_u\lambda_u^{p_{m-2}}(e^u_2+e^u_3)})
	\end{equation}
where we use	\textbf{Lemma~\ref{Dhn cauchy}} and \textbf{Lemma~\ref{alpha_n lambda u}}. 

Thus, by \textbf{(\ref{R7})} (or \textbf{(\ref{R8})}) and \textbf{(\ref{R9})} we get
	\begin{equation*}
	\dfrac{\left| Dh(x)-Dh(y)\right|}{| x-y|^{\beta}}=O\left(e^{\lambda_u^{p_{m-2}}\left[\frac{1}{\overline{\ell}}\overline{c}_u(e^u_2+e^u_3)-\beta\underline{c}_ue^u_2\frac{2\lambda^2_u}{\lambda_u-1}\right]}\right)
	=O(1); 
	\end{equation*}
	for $0<\beta<\dfrac{\overline{c}_u(e^u_2+e^u_3)(\lambda_u-1)}{2\underline{c}_u\overline{\ell}e^u_2 \lambda^2_u}=\dfrac{\overline{c}_u(1+e^u_3)(\lambda_u-1)}{2\underline{c}_u\overline{\ell} \lambda^2_u} $, see point \textbf{5} of  \textbf{Proposition~\ref{super formular}}.
\end{proof}

\subsection{Holder Conjugacy }
\begin{prop}\label{h Holder}
	Let $(\ell_{1},\ell_{2})\in(1,2)^2$ and $\beta=\dfrac{\overline{c}_u(g)}{\underline{c}_u(f)}\dfrac{(e^u_2+e^u_3)(\lambda_u-1)}{2e^u_2 \lambda_u} $.
	If $f,g\in\mathcal{W}_{[1]}$ with critical exponents $(\ell_{1},\ell_{2})$,  then $h$  the topological conjugation between $f$ and $g$ is $C^{\beta}. $
\end{prop}

\begin{proof}
	Let $x,y\in K_f$ and  choose $n$ maximal such that there exists $I(f)\in\mathcal{P}_n $ and $I(f)$ contains $[x,y]$. Since 
	$B_{n+1}^{i+q_{n-1}}(f)=D_n^i(f)$, then by maximality of $n$, $I(f)\neq D_n^i(f)$. Also, as $x,y\in K_f$, then $I(f)\neq C_n^i(f)$. So, either $I(f)=A_n^i(f)$ or  $I(f)=B_n^i(f)$. The assumptions on $Dh$ in the previous section are the same on $ h$ in this section.  Thus, the proof is done as in the previous section where we show that $Dh$ the derivative of $h$ is Holder. For more details, we have:
	
	In the case where $I(f)=A_n^i(f)$, then we can assume that  $x\in D^{i}_{n+1}(f) $ and $y\in B^{i}_{n+1}(f) $. Thus,  by \textbf{Lemma~\ref{cn and the other side}} and  \textbf{Lemma~\ref{I, minoration}}, we obtain
	\begin{equation}\label{R10}
	|x-y|\geq K| A_n^i(f)|\geq \frac{K}{C}e^{\underline{c}_u(f)\lambda_u^{p_{n}}e^u_2\frac{2\lambda_u}{\lambda_u-1}}.
	\end{equation}
	Also, by \textbf{Lemma~\ref{I, majoration}} we get
	\begin{equation}\label{R11}
	|h(x)-h(y)|\leq |h(I(f))|=|h(I(g))| 
	=O(e^{\frac{1}{{\overline{\ell}}}\overline{c}_u(g)\lambda_u^{p_{n}}(e^u_2+e^u_3)}).
	\end{equation}
	By \textbf{(\ref{R10})} and \textbf{(\ref{R11})} we have
	\begin{equation*}
	\dfrac{|h(x)-h(y)|}{| x-y|^{\beta}}=O(e^{\lambda_u^{p_{n}}\left[\overline{c}_u(g)(e^u_2+e^u_3)-\beta\underline{c}_u(f)e^u_2\frac{2\lambda_u}{\lambda_u-1}\right]})
	=O(1)
	\end{equation*}
	for \begin{equation*}0<\beta<\dfrac{ \overline{c}_u(g)}{\underline{c}_u(f)}\dfrac{(e^u_2+e^u_3)(\lambda_u-1)}{2e^u_2 \lambda_u}=\dfrac{ \overline{c}_u(g)}{\underline{c}_u(f)}\dfrac{(1+e^u_3)(\lambda_u-1)}{2 \lambda_u},\end{equation*} see point \textbf{5} of  \textbf{Proposition~\ref{super formular}}.
	
	Now for $I=B_n^i$, then we can choose $(x,y)$ such that $x\in A^{i}_{n+1} $  and 
	$y\in B^{i+q_{n-1}}_{n+1} $.
	
	If $x\in(f^{q_{n+3}+i+1}(U_f), f^{i+1}(U_f)) $ (Let us put $m:=n+1$\mbox{ )} then by  \textbf{Lemma~\ref{cn and the other side}} and  \textbf{Lemma~\ref{I, minoration}}, we obtain
	\begin{equation}\label{R12}
	|x-y| \geq K| A_{m}^i|\geq \frac{K}{C}e^{\underline{c}_u\lambda_u^{p_{m}}e^u_2\frac{2\lambda_u}{\lambda_u-1}}.
	\end{equation}
	
	Otherwise, let $k>0$ $\mbox{(}m:=n+2k\mbox{)} $ maximal such $x\in A^{i+q_{n+1}}_{m}(f) $ and $y\in B^{i+q_{n+1}}_{m}(f) $. By \textbf{Lemma~\ref{I, minoration}} we get
	\begin{equation}\label{R13}
	|x-y|\geq\frac{1}{2} \min\{ | A^{i+q_{n+1}}_{m}(f)|, |  B^{i+q_{n+1}}_{m}(f)|\} \geq\frac{1}{C}e^{\underline{c}_u(f)\lambda_u^{p_{m}}e^u_2\frac{2\lambda_u}{\lambda_u-1}}.
	\end{equation}
	Thus, in the both  cases, by   \textbf{Lemma~\ref{Dhn cauchy}} and \textbf{Lemma~\ref{alpha_n lambda u}} we get
	\begin{equation}\label{R14}
	|h(x)-h(y)|\leq |h(I(f))|=|h(I(g))| 
	=O(e^{\overline{c}_u(g)\lambda_u^{p_{n}}(e^u_2+e^u_3)})
	\end{equation}	
	and by \textbf{(\ref{R12})}  (or \textbf{(\ref{R13})}) and \textbf{(\ref{R14})} we obtain
	\begin{equation*}
	\dfrac{|h(x)-h(y)|}{| x-y|^{\beta}}=O(e^{\lambda_u^{p_{m}}\left[\overline{c}_u(g)(e^u_2+e^u_3)-\beta\underline{c}_u(f)e^u_2\frac{2\lambda_u}{\lambda_u-1}\right]})
	=O(1)
	\end{equation*}
	for \begin{equation*}0<\beta<\dfrac{ \overline{c}_u(g)}{\underline{c}_u(f)}\dfrac{(e^u_2+e^u_3)(\lambda_u-1)}{2e^u_2 \lambda_u}=\dfrac{ \overline{c}_u(g)}{\underline{c}_u(f)}\dfrac{(1+e^u_3)(\lambda_u-1)}{2 \lambda_u},\end{equation*} see point \textbf{5} of  \textbf{Proposition~\ref{super formular}}.

	This concludes the proof. 
	
\end{proof}

\section{ Proof of  Lemma~\ref{main lemma}}\label{Section 6}
We use the formalism presented in \cite{GJSTV} where the authors find a transition between the degenerate geometry (i.e, $\alpha_n$ goes to zero) and the bounded geometry (i.e, $\alpha_n$ is bounded away from zero) for the critical exponents $(\ell,\ell); \ell>1$. These results were generalized in \cite{NTB} for some $(\ell_1,\ell_2)$ belonging to $[1,\infty[$ under assumption $k_1=k_2$, the coefficients before the powers of $x$ in the asymptotics of a map $f$  our class near the ends of the flat interval (see point 1 of \textbf{Remark~\ref{add assump 2}}). More precisely, in \cite{NTB} the author uses this condition in the lemma 2 (p. 658) whose the \textbf{Lemma~\ref{use additional condition}} is a reformulation for Fibonacci circle map with a flat piece.  In the generic case, the condition $k_1=k_2$ will fail for the first return maps, so it does not hold for  infinitely renormalizable maps. This is the first reason why we resume the study of the asymptotic behavior of $\alpha_n$ in our case. This proof can simply be adapted in \cite{NTB}. Other results on the geometry of circle maps with a flat interval can be found in \cite{Gra} and \cite{LB}.

Let us put together  sequences which are frequently used in this section.
\begin{equation*}
\alpha_n=\dfrac{|(\underline{-q_n}, \underline{0})|}{|[\underline{-q_n}
	, \underline{0})|},\;
\sigma_n=	\dfrac{|( \underline{0},\underline{q_n})|}{|( \underline{q_{n-1}}, \underline{0})|}  \mbox{ and } s_n:=\dfrac{|[\underline{-q_{n-2}}  \underline{0}]|}{|\underline{0}|}
\end{equation*}

\subsection{Preliminaries}
\subsubsection{Cross-Ratio Inequalities }		
\begin{nota}	
	We denote by $\R_<^4$, the subset of  $\R^4$ defined by		
	\begin{equation*}
	\R_<^4:= \{ (x_1,x_2,x_3,x_4)\in\R^4,\; such\; that\;x_1<x_2<x_3<x_4 \}.
	\end{equation*}	
\end{nota}
The   following result comes from \textbf{1.5 Lemma} in \cite{de Melo and Van Strien 89}.
\begin{prop}\label{CRI}	The Cross-Ratio Inequality.
	
	 	Let $f\in\mathscr{L}^X$ with $U$ as its flat piece.
Let $(a,b,c,d)\in \R_<^4$. The   cross-ratio \po\, is defined by
	\begin{equation*}
	\po\:(a,b,c,d):=\dfrac{|d-a||b-c|}{|c-a||d-b|}.
	\end{equation*}	
	
	The distortion of the   cross-ratio \po\,  are given  by
	\begin{equation*}
	\D\po\:(a,b,c,d):=\dfrac{\po\:(f(a),f(b),f(c),f(d))}{\po\:(a,b,c,d)}.
	\end{equation*}

	Let us consider a set of $n+1$ quadruples $\{a_i, b_i, c_i, d_i \} $ with the following properties:
	
	\begin{enumerate}
		\item Each point of the circle belongs to at most $k$ intervals $(a_i, d_i)$.
		\item The intervals $(b_i, c_i)$ do not intersect $U$.
	\end{enumerate}
	Then 
	\begin{equation*}
	\prod_{i=0}^{n}\D\po\:(a_i, b_i, c_i, d_i)\geq 1.
	\end{equation*} 
\end{prop}

\begin{rem}	Let $I$ and $J$ be two intervals of finite and nonzero lengths such that $\bar{I}\cap \bar{J}=\emptyset$. We assume that, $J$ is on the right of $I$ and we put $I:=[a,b]$ and $J:=[c,d]$. Then	
	\begin{equation*}
	\po\:(I,J):=\dfrac{|(I,J)||[I,J]|}{|[I,J)||(I,J]|}=\po\:(a,b,c,d).
	\end{equation*}
\end{rem}
\subsubsection{Basic Lemmas}
\begin{lem}\label{2qn and qn} The ratio
	$$ \dfrac{|( \underline{2q_{n-1}}, \underline{q_{n-1}})|}{|( \underline{2q_{n-1}},\underline{0})|}$$  is uniformly bounded away from zero.
\end{lem}
The proof can be found in \cite{GJSTV} (proof of \textbf{Lemma 1.2}). 
\begin{lem}\label{sigma go a away}
	The  sequence 
	\begin{equation*}	
		\dfrac{|( \underline{1},\underline{q_n+1})|}{|( \underline{q_{n-1}+1}, \underline{1}|}	
	\end{equation*}	
	is bounded.	\end{lem}

\begin{proof}
	We will show that the sequence $$\dfrac{|( \underline{q_{n-1}+1}, \underline{1}|}{|( \underline{1},\underline{q_n+1})|}	$$ is uniformly bounded away from zero.	
	Let us observe that, the previous ratio is larger than 
	$$\cro\,(( \underline{q_{n-1}+1}, \underline{1}), \underline{-q_{n-1}+1} )$$
	which by \cri\, with $f^{q_{n-1}-1}$ is greater than
	$$ \dfrac{|( \underline{2q_{n-1}}, \underline{q_{n-1}})|}{|( \underline{2q_{n-1}},\underline{0})|}$$ times a constant. This last ratio is uniformly bounded away from zero, see \textbf{Lemma~\ref{2qn and qn}}.
\end{proof}	

\subsection{A priori Bounds of $\alpha_n$ }
\begin{prop}\label{asymptotically}
	Let $n\in\N$ and $(\ell_1,\ell_2)\in(1,2)^2$. 
	\begin{description}
		\item For all $\alpha_n$, 
		\begin{equation*}
		\alpha_{n}^{\frac{\ell_1}{2},\frac{\ell_2}{2}}<0.55;
		\end{equation*}
		\item for at least every other $\alpha_n$
		\begin{equation*}
		\alpha_{n}^{\frac{\ell_1}{2},\frac{\ell_2}{2}}\leq 0.3.
		\end{equation*}
		\item	If 		\begin{equation*}\alpha_{n-1}^{\frac{\ell_1}{2},\frac{\ell_2}{2}}>0.3 ,
		\end{equation*}	
		then either,
		\begin{equation*}
		\alpha_{n-1}^{\frac{\ell_1}{2},\frac{\ell_2}{2}} <0.44 \quad \mbox{or}\quad \alpha_{n}^{\frac{\ell_1}{2},\frac{\ell_2}{2}}<0.16.
		\end{equation*}	
		
	\end{description}	
\end{prop}

\begin{proof}
Let 
\begin{equation*}
\gamma_{1,n}=|(\underline{-q_{n}}, \underline{0})|,\quad  \gamma_n = \dfrac{\gamma_{1,n}}{\gamma_{1,n-1}},
\end{equation*}
\begin{equation*}\gamma_n^{(\ell_1|\ell_2)}:=\frac{k_2}{k_1}\cdot
\dfrac{\gamma^{\eg}_{1,n}}{\gamma^{\ed}_{1,n-1}}  \mbox{  if }  n\in 2\Z \mbox{ and }\gamma_n^{(\ell_1|\ell_2)}:=\frac{k_1}{k_2}\cdot\dfrac{\gamma^{\ed}_{1,n}}{\gamma^{\eg}_{1,n-1}}
\mbox{ if }  n\in 2\Z +1
\end{equation*}
where $k_1$ and $k_2$ come from \textbf{Fact~\ref{fact}}.
			
	These notations simplify the formalization of the following lemma which will play an important (essential) role in the proof of \textbf{Proposition~\ref{asymptotically}}.
	\begin{lem}\label{use additional condition}
		For every $n\in\N$ large enough, the following inequality holds  
		\begin{equation}\label{lemma inequality 1}
		\dfrac{(\alpha^{\ell_1,\ell_2}_{n}+ \alpha^{\ell_1,\ell_2}_{n-1}\gamma_n^{(\ell_1|\ell_2)}  )(1+\gamma_n^{(\ell_1|\ell_2)}  )}{(1+ \alpha^{\ell_1,\ell_2}_{n-1}\gamma_n^{(\ell_1|\ell_2)})(\alpha^{\ell_1,\ell_2}_{n}+\gamma_n^{(\ell_1|\ell_2)} )  }\leq s_n\alpha_{n-2}.
		\end{equation}		
	\end{lem}
	\begin{proof} Let  $n\in 2\N+1$  large enough.
	By point 1 of \textbf{Remark~\ref{add assump 2}}, the left hand side of 	\textbf{(\ref{lemma inequality 1})} is equal to the cross-ratio 	
		\begin{equation*}
		\po\,(\underline{-q_{n}+1}, \underline{-q_{n-1}+1}).  
		\end{equation*}
		Applying $f^{q_{n-1}-1} $, by the expanding cross-ratio property, we get the inequality.
	\end{proof}

	The left hand side  is a function of the three	variables $\alpha^{\ell_1,\ell_2}_{n} $, $\alpha^{\ell_1,\ell_2}_{n-1} $ and $\gamma_n^{(\ed|\eg)}.$ 
	Observe that the function increases monotonically with each of the first two variables. However, relatively to the third variable, the function reaches a minimum. To see this, take the logarithm of the function and
	check that the first derivative is equal to zero only when	
	\begin{equation*}
	(\gamma_n^{(\ed|\eg)})^2=\dfrac{\alpha^{\ell_1,\ell_2}_{n}}{\alpha^{\ell_1,\ell_2}_{n-1}}.
	\end{equation*}	
	By substituting this for $\gamma_n^{(\ed|\eg)}$ we get that
	\begin{equation}\label{quadradic inquality}
	\left(\dfrac{\alpha^{\frac{\ell_1}{2},\frac{\ell_2}{2}}_{n}+ \alpha^{\frac{\ell_1}{2},\frac{\ell_2}{2}}_{n-1}}{ 1+\alpha^{\frac{\ell_1}{2},\frac{\ell_2}{2}}_{n} \alpha^{\frac{\ell_1}{2},\frac{\ell_2}{2}}_{n-1}}\right)^2\leq s_n\alpha_{n-2}.
	\end{equation}
	Let be the sequence ${y'}_{n} $ defined by
	\begin{equation*} {y'}_{n}=\min\{\alpha^{\frac{\ell_1}{2},\frac{\ell_2}{2}}_{n}, \alpha^{\frac{\ell_1}{2},\frac{\ell_2}{2}}_{n-1}    \}	.	\end{equation*}
	
	Since $\alpha_{n-2}\leq \alpha^{\frac{\ell_1}{2},\frac{\ell_2}{2}}_{n-2}$. Substituting the above variable into
	inequality \textbf{(\ref{quadradic inquality})} gives rise to a quadratic inequality ${y'}_{n} $ whose only root in the interval $(0,1)$ is given by
	\begin{equation*} \dfrac{ \sqrt{s_n \alpha^{\frac{\ell_1}{2},\frac{\ell_2}{2}}_{n-2} }}{1+\sqrt{1-s_n \alpha^{\frac{\ell_1}{2},\frac{\ell_2}{2}}_{n-2} }}.	\end{equation*}
	That is,
	\begin{equation}\label{inequation min} {y'}_{n}=\min\{\alpha^{\frac{\ell_1}{2},\frac{\ell_2}{2}}_{n}, \alpha^{\frac{\ell_1}{2},\frac{\ell_2}{2}}_{n-1}    \}\leq\dfrac{ \sqrt{s_n \alpha^{\frac{\ell_1}{2},\frac{\ell_2}{2}}_{n-2} }}{1+\sqrt{1-s_n \alpha^{\frac{\ell_1}{2},\frac{\ell_2}{2}}_{n-2} }}.	\end{equation}
	We use the elementary lemma (\textbf{Lemma 3.4} in \cite{GJSTV}).	
	\begin{lem} \label{function h}
		The function 
		\begin{equation*} 
		h_n(z)=\dfrac{ \sqrt{s_n z }}{1+\sqrt{1-s_n z }}		
		\end{equation*}
		moves points to the left, $h(z)<z$, if $z\geq 0.3$ and $n$ is large enough. 
	\end{lem}

	\begin{lem}\label{subsequence}
		There is a subsequence of $\alpha_n$ including at least every other $\alpha_n$,
		such that 	
		\begin{equation*}	\limsup \alpha_{n}^{\frac{\ell_1}{2},\frac{\ell_2}{2}}\leq 0.3.	\end{equation*}	
		
	\end{lem}
	\begin{proof}
		We select the subsequence.
		\begin{enumerate}
			\item \textbf{The initial term:}
			There exists $n-2\in\N$, such that $\alpha_{n-2}^{\frac{\ell_1}{2},\frac{\ell_2}{2}}\leq 0.3$. This comes directly from the properties of the function $h_n$ (\textbf{Lemma~\ref{function h}}) and inequality \textbf{(\ref{inequation min})}. 
			\item \textbf{The next element}. Suppose  that $\alpha_{n-2} $ has been selected. If 
			\begin{equation*}
			{y'}_{n}=\min\{\alpha_{n}^{\frac{\ell_1}{2},\frac{\ell_2}{2}}, \alpha_{n-1} ^{\frac{\ell_1}{2},\frac{\ell_2}{2}}   \}=\alpha_{n-1} ^{\frac{\ell_1}{2},\frac{\ell_2}{2}}\quad\mbox{or}\quad \alpha_{n-1} ^{\frac{\ell_1}{2},\frac{\ell_2}{2}} \leq 0.3,	\end{equation*}
			then, we select $\alpha_{n-1} $ as the next term. Otherwise, $\alpha_{n} $
			is the next term. Thus, the sequence is constructed.
		\end{enumerate}
		
	\end{proof}
	
	\begin{cor}
		For the whole sequence $\alpha_n$ we have 
		\begin{equation*}	\limsup \alpha_{n}^{\frac{\ell_1}{2},\frac{\ell_2}{2}}\leq (0.3)^{\frac{1}{2}}. 	\end{equation*}	
		Moreover, 
		if $\alpha_{n-1} $ does belong to the subsequence $\alpha_n$  defined by 	\textbf{Lemma~\ref{subsequence}} then either
		\begin{equation*}	\alpha_{n-1}^{\frac{\ell_1}{2},\frac{\ell_2}{2}}<0.44\quad \mbox{or}\quad 	\alpha_{n}^{\frac{\ell_1}{2},\frac{\ell_2}{2}}<0.16. 	\end{equation*}		
	\end{cor}
	\begin{proof}
		Observe that the function 
		\begin{equation*}
		H:(s,t)\in\R_+^2\mapsto F(s,t)=\dfrac{s+t}{1+st}
		\end{equation*}
		is symmetric and for fixed $s$, the function $F(s,\cdot)$
		reaches its minimum at zero by taking the value $s$. Therefore, for every $s,t\geq 0$,
		\begin{equation*}
		s,t\leq\dfrac{s+t}{1+st}. 
		\end{equation*}
		So,
		\begin{equation} \label{inquality alpha_n alpha_n-1}
		\alpha_{n}^{\frac{\ell_1}{2},\frac{\ell_2}{2}}, \;\alpha_{n-1}^{\frac{\ell_1}{2},\frac{\ell_2}{2}}\leq \dfrac{\alpha_{n}^{\frac{\ell_1}{2},\frac{\ell_2}{2}}+ \alpha_{n-1}^{\frac{\ell_1}{2},\frac{\ell_2}{2}}}{ 1+\alpha_{n}^{\frac{\ell_1}{2},\frac{\ell_2}{2}} \alpha_{n-1}^{\frac{\ell_1}{2},\frac{\ell_2}{2}}}.
		\end{equation}
		Thus, according to that $\alpha_{n-2} $ is an element of the sequence of \textbf{Lemma~\ref{subsequence}}, it follows from \textbf{(\ref{quadradic inquality})}  that the right member of \textbf{(\ref{inquality alpha_n alpha_n-1})}  is estimated as follows
		\begin{equation}\label{estimation ratio of quadratic}
		\dfrac{\alpha_{n}^{\frac{\ell_1}{2},\frac{\ell_2}{2}}+ \alpha_{n-1}^{\frac{\ell_1}{2},\frac{\ell_2}{2}}}{ 1+\alpha_{n}^{\frac{\ell_1}{2},\frac{\ell_2}{2}} \alpha_{n-1}^{\frac{\ell_1}{2},\frac{\ell_2}{2}}}\leq\sqrt{s_n\alpha_{n-2}^{\frac{\ell_1}{2},\frac{\ell_2}{2}}}\approx \sqrt{\alpha_{n-2}^{\frac{\ell_1}{2},\frac{\ell_2}{2}}} \leq\sqrt{ 0.3}.
		\end{equation}
		
		Now, suppose that $\alpha_{n-1} $ does not belong  to the    subsequence chosen in the proof
		of \textbf{Lemma~\ref{subsequence}}, then 
		\begin{equation}\label{alphan is min}
		\min\{\alpha_{n}^{\frac{\ell_1}{2},\frac{\ell_2}{2}}, \alpha_{n-1} ^{\frac{\ell_1}{2},\frac{\ell_2}{2}}   \}=\alpha_{n}^{\frac{\ell_1}{2},\frac{\ell_2}{2}}\leq 0.3.	\end{equation}
		Thus, if  $\alpha_{n}^{\frac{\ell_1}{2},\frac{\ell_2}{2}}\geq 0.16,  $
		then by combining this with \textbf{(\ref{estimation ratio of quadratic})} and \textbf{(\ref{alphan is min})}, we obtain the desired estimate.   \end{proof}
	
	This ends the prove of  \textbf{Proposition~\ref{asymptotically}}.
\end{proof}

\subsection{Recursive formula of  $\alpha_n$ }

\begin{prop}\label{recurrence formula}
Fix $\ell_{1},\,\ell_{2}>1$.	Let $n $ be an integer large enough,
 ,  we have  
		\begin{equation} \label{inequation recurrence formula}
		\alpha_{2n}^{\ell_{1}}\leq M_{2n}(\ell_{1})\alpha^{2}_{2n-2}\quad\mbox{and} \quad
		\alpha_{2n+1}^{\ell_{2}}\leq M_{2n+1}(\ell_{2})\alpha^{2}_{2n-1}
		\end{equation}		
		where\begin{equation*}
		M_n(\ell)=s^2_{n-1}\cdot \dfrac{2}{\ell}\cdot \dfrac{1}{1+\sqrt{1-\dfrac{2(\ell-1)}{\ell}s_{n-1}\alpha_{n-1}}}
		\cdot \dfrac{1}{1-\alpha_{n-2}}\cdot\frac{\sigma_{n}}{\sigma_{n-2}} .
		\end{equation*}
\end{prop}

\begin{proof} We treat the case even $n$ and the case odd $n$  is treated in a similar way. Recall that, 
	\begin{equation*}
	\alpha_n=\dfrac{|(\underline{-q_n}, \underline{0})|}{|[\underline{-q_n}
		, \underline{0})|}.
	\end{equation*}
	For every even $n$  and  large enough,  applying $f$ to the equality we have
	\begin{equation*}
	\alpha^{\ell_{1}}_n=\dfrac{|(\underline{-q_n+1}, \underline{1})|}{|[\underline{-q_n+1}
		, \underline{1})|}
	\end{equation*}
	which is certainly less than the cross-ratio
	\begin{equation*}
	\po(\underline{-q_n+1}, (\underline{1},\underline{-q_{n-1}+1}]).
	\end{equation*}
	Since the cross-ratio \po\, is expanded by $f^{q_{n-1}-1} $, then
	\begin{equation}\label{rn4}
	\alpha^{\ell_{1}}_n< \delta_n(1)s_n(1)
	\end{equation}
	with
	\begin{equation*}
	\delta_n(k):=\dfrac{|(\underline{-q_{n}+kq_{n-1}}, \underline{kq_{n-1}})|}{|[\underline{-q_{n}+kq_{n-1}}
		, \underline{kq_{n-1}})|}
	\end{equation*}
	and 
	\begin{equation*}
	s_n(k):=\dfrac{|[\underline{-q_{n}+kq_{n-1}}, \underline{0}]|}{|(\underline{-q_{n}+kq_{n-1}}
		, \underline{0}]|}.
	\end{equation*}
	By multiplying and dividing the right member of \textbf{(\ref{rn4})} by 
	$\alpha^2_{n-2} ,$ we get.
	\begin{equation}
	\alpha^{\ell_{1}}_n\leq s_n\nu_{n-2} \mu_{n-2} \alpha^2_{n-2}\label{rn5}
	\end{equation}
	with
	\begin{equation*}
	\nu_{n-2}:=\dfrac{|[\underline{-q_{n-2}}, \underline{0})|}{|(\underline{q_{n-1}}
		, \underline{0})|} \cdot \dfrac{|[\underline{-q_{n-2}}, \underline{0})|}{|[\underline{-q_{n-2}}
		, \underline{q_{n-1}})|}
	\end{equation*}
	and
	\begin{equation*}
	\mu_{n-2} :=\dfrac{|(\underline{-q_{n-2}}, \underline{q_{n-1}})|}{|(\underline{-q_{n-2}}
		, \underline{0})|}.
	\end{equation*}
	It remains to estimate $\nu_{n-2}$ and $\mu_{n-2}$ to end this part.
	For $\nu_{n-2}$, observe that
	\begin{equation*}
	|(\underline{-q_{n-2}}, \underline{0})|\leq |(\underline{q_{n-3}}, \underline{0})|
	\end{equation*}
	so that,
	\begin{equation}\label{rn6}
	\nu_{n-2}=\dfrac{    1  }{\sigma_{n-1} \sigma_{n-2}}\cdot \dfrac{1}{1-\alpha_{n-2}}
	\end{equation}	
	The estimate $\mu_{n-2}$ we use the following  lemma (\textbf{Lemma 3.1}in \cite{GJSTV}).
	\begin{lem}\label{rn7}
		Let $\ell\in(1,2)$. For all numbers $x>y$, we have the following inequality:
		\begin{equation*}
		\dfrac{x^\ell-y^\ell}{x^\ell}\geq \left( \dfrac{x-y}{x}\right)\left[\ell- \dfrac{\ell(\ell-1)}{2}\left(\dfrac{x-y}{x}\right)\right].
		\end{equation*}
		
	\end{lem}
	
	Now, apply $f$ into the intervals defining the ratio $\mu_{n-2} $. By \textbf{Lemma~\ref{rn7}}, the resulting ratio is larger than

	\begin{equation*}
	\mu_{n-2}( \ell_{1}-\dfrac{\ell_{1}(\ell_{1}-1)}{2} \mu_{n-2}).
	\end{equation*}
	The cross-ratio \po
	\begin{equation*}
	\po\:(-q_{n-2}+1, (\underline{q_{n-1}+1}, \underline{1})).
	\end{equation*}
	That is,
	\begin{equation*}
	\dfrac{|(\underline{-q_{n-2}+1},  \underline{q_{n-1}+1})| |[\underline{-q_{n-2}+1}, \underline{1})|}
	{|[\underline{-q_{n-2}+1},  \underline{q_{n-1}+1})| |(\underline{-q_{n-2}+1}, \underline{1})|}
	\end{equation*}
	is larger again. Thus, by expanding cross-ratio  property on $f^{q_{n-2}} $, we obtain:
	\begin{equation*}\label{rn8}
	\mu_{n-2}( \ell_{1}-\dfrac{\ell_{1}(\ell_{1}-1)}{2} \mu_{n-2})\leq s_{n-1}\sigma_{n}\sigma_{n-1}.
	\end{equation*}
	By solving this quadratic inequality, we obtain
\begin{equation}\label{rn11}
\mu_{n-2}<\dfrac{2}{\ell_{1}}\cdot \dfrac{1}{1+\sqrt{1-\dfrac{2(\ell_{1}-1)}{\ell_{1}} s_{n-1}\sigma_{n}\sigma_{n-1}}}
s_{n-1}\sigma_{n}\sigma_{n-1}.
\end{equation}
Since, $\sigma_{n}\sigma_{n-1}<\alpha_{n-1} $, the first inequality in \textbf{(\ref{inequation recurrence formula})}  follows  by combining  inequalities	
	\textbf{(\ref{rn5})}, \textbf{(\ref{rn6})} and \textbf{(\ref{rn11})}.
\end{proof}

  $\alpha_n$ goes to zero

Technical reformulation of \textbf{Proposition~\ref{recurrence formula}}. Let $W_n$ be a sequence  defined by
\begin{equation*}
M_n(\ell)=W_n(\ell)\dfrac{\sigma_{n}}{\sigma_{n-2}}.
\end{equation*}
Let 
\begin{equation*}
M'_n(\ell): =M_n(\ell)\alpha_{n-2}^{2-\ell} \quad\mbox{and}\quad W'_n(\ell): =W_n(\ell)\alpha_{n-2}^{2-\ell}.
\end{equation*}
For every $n$ even large enough, the recursive formula \textbf{(\ref{inequation recurrence formula})} can be written  in the form:
\begin{equation*}
\alpha_{n}^{\ell_1} \leq W'_n(\ell_1)\dfrac{\sigma_{n}}{\sigma_{n-2}}\alpha_{n-2}^{\ell_1}.
\end{equation*}
so, 
\begin{equation*}
\alpha_{n}^{\ell_1} \leq \prod_{k=2}^{k=n} W'_k(\ell_1)\dfrac{\sigma_{n}}{\sigma_{0}}\alpha_{0}^{\ell_1}.
\end{equation*}
\subsubsection*{ $\prod_{k=2}^{k=n} W'_k(\ell_1)$ goes to zero.}
Observe that the size of $	W'_n(\eg)$ is given by the study of the function	
\begin{equation*}
W'_n(x,y,\eg)= \dfrac{1}{	\dfrac{\eg}{2}+	\dfrac{\eg}{2}\sqrt{1-\dfrac{2(\eg-1)}{\eg}x^{\frac{2}{\ed}}}}
\cdot \frac{y^{\frac{4}{\eg}-2}}{1-y^{\frac{2}{\eg}}}.
\end{equation*}
The meaning of variation of $W'_k(x,y,\ell_1)$ relative to the third variable is given by the following lemma (\textbf{Lemma 3.2} in \cite{GJSTV}).
\begin{lem}\label{lemma for conclusion}
	For any $0<y<\frac{1}{\sqrt{e}}$, $x\in(0,1)$ and $\eg\in(1,2]$  the function $	W'_n(x,y,\eg)$ is increasing with respect to $\eg$.
\end{lem}

\subsubsection*{Analyse the asymptotic size of $W'_i(2)$. }
Since the hypotheses of \textbf{Lemma~\ref{lemma for conclusion}} are satisfied (\textbf{Proposition~\ref{asymptotically}}),
it enough to verify that the convergence  of $\prod_{i=1}^{n}W'_i(2)$. 

\begin{description}
	\item[\textbf{-}]If $\alpha_{ n-2} < (0.3)^{\eg}$, then 
	$W'(2)<W'(0.55,0.16,2)<0,9$.
	\item[\textbf{-}] If not, then by \textbf{Proposition~\ref{asymptotically}}, 	$W'(2)<W'(0.3,0.44,2)<0,98$
	or else,  $W_{n+1}'(2)W_{n}'(2)<W'(0.55,0.16,2)W'(0.16,0.55,2)<0,85$.
\end{description}
As a consequence, we have
\begin{cor}\label{alpha_n go to zero double exponentially}
	Let $\ell_1,\ell_2\in (1,2)$. Then
	$\alpha_{n} $  go to zero least double exponentially fast. 
\end{cor}

\textbf{Acknowledgements:}

I would sincerely thank Prof. M. Martens and Prof. Dr. L. Palmisano for introducing me to the subject of this paper, his valuable advice and helpful discussions.


\begin{thebibliography}{99}
\bibitem{AV}{ Artur A.}, On Rigidity of Critical Circle, B. Braz. Math. Soc.,  2013, Vol.44, pp.611-619.
\bibitem{FM}{de Faria, E., and de Melo, W.}, Rigidity of critical circle mappings, Soc. 2000,  Vol.13, pp.343-370.
\bibitem{FMP}{ de Faria, E., de Melo, W. and Pinto, A.}, Global hyperbolicity of renormalization for
$C^r$ unimodal mappings, Ann. of Math., 2006,   Vol.164, pp.731-824.
\bibitem{MP}{ de Melo, W., and Pinto, A. A.}, Rigidity of $C^2$ infinitely renormalizable unimodal
maps, Comm. Math. Phys., 208, 1999, 1, 91-105.
\bibitem{de Melo and Van Strien 89}{ de Melo, W., and Van Strien, S.},
One-Dimensional Dynamics: The Schwarzian Derivative and Beyond, Annals of Mathematics.,  1989, Vol.129, pp.519-546.
\bibitem{de Melo and van Strien}{ de Melo W., and Van Strien, S.}, One-Dimensional Dynamics., Springer-Verlag, 1993.
\bibitem{GMM}{ Guarino, P.,  Martens, M. and de Melo, W.},  Rigidity of critical circle maps, Duke Math. J.,  2018, Vol.167, pp.2125-2188.
\bibitem{Gra}{ Graczyk, J.}, Dynamics of circle maps with flat spots, Fundamenta. Mathematicae.,  2010, Vol.209, pp.267-290.
\bibitem{GJSTV}{ Graczyk, J., Jonker, L. B., \'{S}wiatek, G., Tangerman, F. M. and Veerman, J. J. P.}, Differentiable Circle Maps with a Flat Interval, Commun. Math. Phys., 1995, Vol.173,  pp.599-622.
\bibitem{He}{ Herman, M.}, Sur la Conjugaison diff\'{e}rentiable des diff\'{e}omorphimes du cercle \`{a} des rotation , Inst. Hautes \'{E}tude Sci. Publ. Math., 1979, Vol.49, pp.5-233.
\bibitem{KK}{Khanin, K., and Koci\'{c}, S.}, Renormalization conjecture and rigidity theory for circle diffeomorphisms with breaks, Geom. and Func. Anal.  2014, Vol.24, pp.2002-2028.
\bibitem{KT}{  Khanin, K. and Teplinsky, A.}, Robust rigidity for circle diffeomorphisms with singularities, Invent. Math.  2007, Vol.169, pp.193-218.
\bibitem{ML}{ Martens, M., and Palmisano, L.}, Foliation by Rigigity Class,  arXiv: 1704.06328v1, (24 Apr. 2017).
\bibitem{MSM}{Martens, M.,  Strien, S., Melo,  W.,  and  Mendes, P.,} On Cherry flows, Erg. Th. and Dyn. Sys.,  1990,  Vol.10, pp.531-554.
\bibitem{Mc}{ McMullen, C. T.}, Renormalization and 3-manifolds which fiber over the circle, Ann. Math. Stud., Vol.142, Princeton: Princeton University Press, 1996.
\bibitem{Mi}{Misiurewicz, M.,} Rotation interval for a class of maps of the real line into itself, Erg. Th. and Dyn. Sys.  1986, Vol. 6, pp.17-132.
\bibitem{MO}{  Mostow, G. D.}, Quasi-conformal mappings in $n$-space and the rigidity of hyperbolic
space forms, Inst. Hautes \'{E}tude Sci. Publ. Math.,  1968, Vol.34, pp.53-104.
\bibitem{NTB}{Ndawa Tangue B.,} Cherry Maps with Different Critical Exponents: Bifurcation of Geometry, Rus. J. Nonlin. Dyn., 2020, Vol. 16, no. 4, pp.  651-672.
\bibitem{Livi2013}{Palmisano, L.,} A Phase Transition for circle Maps and Cherry Flows.,  Commun. Math. Phys.,  2013, Vol.321, pp.135-155.
\bibitem{LB}{ Palminsano L. and Tangue, B.}, A Phase Transition for Circle Maps with a Flat Spot and Different Critical Exponents, arXiv: 1907.10909v1 (22 Jul. 2019).
\bibitem{2}{  \c{S}wi\'{a}tek, G.,} Rational rotation numbers for maps of the circle, Comm. Math. Phys.  1988, Vol.119, pp.109-128.
\bibitem{Ya}{ Yampolsky, M.}, Hyperbolicity of renormalization of critical circle maps, Inst. Hautes \'{E}tude Sci. Publ. Math.,  2002, Vol.96, pp.1-41.
\bibitem{YY}{ Yilun, S., and Yamei, Y.}, Fixed-time group tracking control with unknown inherent nonlinear dynamics,  IEEE Access,  2017, Vol.5, pp.12833 - 12842.
\bibitem{Yo}{ Yoccoz, J.-C.}, Conjugaison diff\'{e}rentiable des diff\'{e}omorphimes du cercle dont le nombre de rotation v\'{e}rifie une condiontion diophantienne, Ann. Sci. \'{E}cole Norm. Sup (4), 1984,  Vol.17,  pp. 333-359,
\bibitem{Yo1}{ Yoccoz, J.-C.}, Il n'y a pas de contre-exemple de Denjoy analytique, C.R. Acad. Paris 298, s\'{e}rie I, 1984, pp.141-144.

 	
 \end{thebibliography}
\end{document}